\documentclass[12pt]{article}
\title{\textbf{ \Large
Crystals of Lakshmibai-Seshadri paths
and extremal~weight modules
over quantum hyperbolic Kac-Moody algebras of rank $2$
}}
\author{Ryuta Hiasa\\
	{\small Graduate School of Pure and Applied Sciences, University of Tsukuba,}\\
	{\small 1-1-1 Tennodai, Tsukuba, Ibaraki 305-8571, Japan}\\
	{\small (e-mail: \texttt{hiasa@math.tsukuba.ac.jp})}
}
\date{}
\usepackage{amsmath,amssymb,amsthm}
\usepackage{enumerate}
\usepackage{mathtools}
\mathtoolsset{showonlyrefs=true}
\newcommand{\Fg}{\mathfrak{g}}

\newcommand{\C}{\mathbb{C}}
\newcommand{\R}{\mathbb{R}}

\newcommand{\Z}{\mathbb{Z}}

\newcommand{\derep}{\Delta_{\mathrm{re}}^+}
\newcommand{\B}{\mathbb{B}}
\newcommand{\CB}{\mathcal{B}}
\newcommand{\wt}{\mathrm{wt}}
\newcommand{\e}{\tilde{e}}
\newcommand{\f}{\tilde{f}}
\newcommand{\ep}{\varepsilon}
\newcommand{\ph}{\varphi}
\newcommand{\al}{\alpha}
\newcommand{\alc}{\alpha^{\vee}}
\newcommand{\0}{\mathbf{0}}
\newcommand{\tpq}[1]{q_{#1}/p_{#1}}
\newcommand{\dpq}[1]{\frac{q_{#1}}{p_{#1}}}
\newcommand{\bra}[1]{\left({#1}\right)}
\newcommand{\CT}{\mathcal{T}}
\newcommand{\resp}[1]{(\text{resp., }#1)}
\newcommand{\pair}[2]{\langle #1, #2 \rangle}
\newcommand{\mor}{\Phi^\lambda_\iota}
\newcommand{\bigset}[2]{\left\{{#1} \left|\ {#2}\right.\right\}}
\newcommand{\img}{\mathrm{Im}}
\newtheorem{theorem}{Theorem}[section]
\newtheorem{proposition}[theorem]{Proposition}
\newtheorem{lemma}[theorem]{Lemma}
\newtheorem{corollary}[theorem]{Corollary}
\theoremstyle{definition}
\newtheorem{definition}[theorem]{Definition}

\newtheorem{remark}[theorem]{Remark}

\numberwithin{equation}{section}
\usepackage[dvipdfm, top=28truemm, bottom=28truemm, left=28truemm, right=28truemm]{geometry}
\usepackage{amscd}
\begin{document}
\maketitle
\begin{abstract}
	Let $\mathfrak{g}$ be a hyperbolic Kac-Moody algebra of rank $2$,
	and let $\lambda$ be an arbitrary integral weight.
	We denote by $\mathbb{B}(\lambda)$ the crystal of all Lakshmibai-Seshadri paths of shape $\lambda$.
	Let $V(\lambda)$ be the extremal weight module of extremal weight $\lambda$
	generated by the (cyclic)
	extremal weight vector $v_\lambda$ of weight $\lambda$,
	and let $\mathcal{B}(\lambda)$ be the crystal basis of $V(\lambda)$ with  $u_\lambda \in \mathcal{B}(\lambda)$ the element corresponding to $v_\lambda$.
	We prove that
	the connected component $\mathcal{B}_0(\lambda)$ of $\mathcal{B}(\lambda)$ containing $u_\lambda$
	is isomorphic, as a crystal, to
	the connected component $\mathbb{B}_0(\lambda)$ of $\mathbb{B}(\lambda)$ containing the
	straight line $\pi_\lambda$.
	Furthermore,
	we prove that if $\lambda$ satisfies a special condition,
	then the crystal basis $\mathcal{B}(\lambda)$ is isomorphic,
	as a crystal, to the crystal $\mathbb{B}(\lambda)$.
	As an application of these results,
	we obtain an algorithm for computing the number of elements of weight $\mu$ in  $\mathcal{B}(\Lambda_1-\Lambda_2)$,
	where $\Lambda_1, \Lambda_2$ are the fundamental weights,
	in the case that $\mathfrak{g}$ is symmetric.
\end{abstract}

\setlength{\baselineskip}{16pt}
\section{Introduction.}
Let $\Fg$ be a symmetrizable Kac-Moody algebra over $\C$,
and  $U_q(\Fg)$ the quantized universal enveloping algebra over $\C(q)$
associated to $\Fg$.
We denote by $W$ the Weyl group of $\Fg$.
Let $P$ be an integral weight lattice of $\Fg$,
and $P^+$ \resp{$-P^+$} the set of dominant \resp{antidominant} integral weights in $P$.
Let $\mu \in P$ be an arbitrary integral weight.
The extremal weight module $V(\mu)$ of extremal weight $\mu$ is
the integrable $U_q(\Fg)$-module generated by a single element $v_\mu$
 with the defining relation that
 $v_\mu$ is an extremal weight vector of weight $\mu$ in the sense
of \cite{K}.
This module was introduced by Kashiwara \cite{K}
as a natural generalization of integrable highest (or lowest) weight modules;
in fact,
if $\mu \in P^+$ \resp{$\mu \in -P^+$}, then
the extremal weight module of extremal weight $\mu$ is
isomorphic, as a $U_q(\Fg)$-module, to  the integrable highest \resp{lowest} weight module
of highest \resp{lowest} weight $\mu$.
Also, he proved in \cite[Proposition 8.2.2]{K} that $V(\mu)$ has a crystal basis $\CB(\mu)$ for all $\mu \in P$;
let $u_\mu$ denote the element of $\CB(\mu)$
corresponding to  $v_\mu \in V(\mu)$.
We know from \cite{K} that $V(\mu)\cong V(w\mu)$ as $U_q(\Fg)$-modules,
and $\CB(\mu)\cong \CB(w\mu)$ as crystals for all $\mu \in P$ and $w \in W$.
Hence we are interested in
 the case that
\begin{align}\label{A}
	W\mu  \cap (P^+ \cup -P^+) = \emptyset.
\end{align}


If $\Fg$ is of finite type, then $W\mu\cap P^+ \neq \emptyset$ for any $\mu \in P$.
Assume that $\Fg$ is of affine type.
Then,
$W\mu  \cap (P^+ \cup -P^+) = \emptyset$
if and only if ($\mu \neq 0$, and)  $\mu$ is of level zero. Naito and Sagaki proved in
\cite{NS1} and \cite{NS2} that
if $\mu$ is a positive integer multiple of a level-zero fundamental weight,
then the crystal basis $\CB(\mu)$ of the extremal weight module $V(\mu)$  is
isomorphic, as a crystal, to the crystal  $\B(\mu)$ of Lakshmibai-Seshadri (LS for short) paths,
which was introduced by Littelmann in \cite{L2} and \cite{L};
see \S \ref{section.LSpath} for the details.
After that, Ishii, Naito, and Sagaki \cite{INS} introduced the notion of semi-infinite LS paths of shape $\mu$ for a level-zero
dominant integral weight $\mu$, and proved that
the crystal basis $\CB(\mu)$ of the extremal weight module $V(\mu)$
is isomorphic, as a crystal, to
the crystal $\B^{\frac{\infty}{2}}(\mu)$ of semi-infinite LS paths of
shape $\mu$.
Now, we assume that $\Fg$ is the hyperbolic Kac-Moody algebra associated to
the generalized Cartan matrix
\begin{align}
	A=
	\begin{pmatrix}
		2 & -a_1 \\-a_2 & 2\\
	\end{pmatrix}, \quad
	\text{ where }
	a_1, a_2 \in \Z_{\geq 2}
	\text{ with }
	a_1 a_2>4.
\end{align}
Yu \cite{Yu} proved that  $\Lambda_1-\Lambda_2 \in P$ satisfies  condition \eqref{A},  where $\Lambda_1, \Lambda_2$ are the fundamental weights,
and that (the crystal graph of) $\B(\Lambda_1-\Lambda_2)$ is connected.
Then, Sagaki and Yu \cite{SY} proved that $\B(\Lambda_1-\Lambda_2)$ is isomorphic, as a crystal, to the crystal basis $\CB(\Lambda_1-\Lambda_2)$ of the
extremal weight module $V(\Lambda_1-\Lambda_2)$ of extremal weight $\Lambda_1-\Lambda_2$.
In \cite{H},
the author obtained the following necessary and sufficient condition for an integral weight to satisfy condition \eqref{A}:
Let $\mathbb{O}:=\{ W\mu \mid \mu \in P \}$ be
	the set of $W$-orbits in $P$.
	A $W$-orbit $O\in \mathbb{O}$ satisfies
	condition \eqref{A}, 	that is,
	$O  \cap (P^+ \cup -P^+) = \emptyset$  if and only if
	$O$ contains
	an integral weight of the form either \eqref{enu.1int} or \eqref{enu.2int}:
	\begin{enumerate}[\upshape(i)]
		\item $k_1\Lambda_1-k_2\Lambda_2$ for some
		$k_1, k_2 \in \Z_{>0}$ such that $k_2\leq k_1<(a_1-1)k_2$;\label{enu.1int}
		\item $k_1\Lambda_1-k_2\Lambda_2$ for some
		$k_1, k_2 \in \Z_{>0}$ such that $k_1<k_2\leq (a_2-1)k_1 $.\label{enu.2int}
	\end{enumerate}
Also, he proved that for $\lambda=k_1\Lambda_1-k_2\Lambda_2 \in P$
of the form  either \eqref{enu.1int} or  \eqref{enu.2int} above,
the crystal graph $\B(\lambda)$ is connected if and only if $k_1=1$ or $k_2=1$.

In this paper, we study
the relationship between the crystal $\B(\lambda)$ of LS paths of shape $\lambda$ and
the crystal basis $\CB(\lambda)$ of extremal weight module $V(\lambda)$ of extremal weight $\lambda$
in the case that
$\lambda=k_1\Lambda_1-k_2\Lambda_2$
is of the form either
\eqref{enu.1int} or \eqref{enu.2int} above.
We prove the following theorems.
\begin{theorem}[$=$ Theorem \ref{thm.mainb0}]\label{thm.mainb0int}
	Let  $\B_0(\lambda)$ \resp{$\CB_0(\lambda)$} be the connected component of $\B(\lambda)$ \resp{$\CB(\lambda)$} containing $\pi_\lambda:=(\lambda; 0,1 )$ \resp{$u_\lambda$}.
	There exists an isomorphism $\B_0(\lambda) \rightarrow \CB_0(\lambda)$
	of crystals that sends $\pi_\lambda$ to $u_\lambda$.
\end{theorem}
\begin{theorem}[$=$ Theorem \ref{thm.mainb}]\label{thm.mainbint}
	If $k_1=1$ or $k_2=1$, that is,
	$\lambda \in P$ is of the form either
	$k_1\Lambda_1-\Lambda_2$ with $1\leq k_1 < a_1-1$ or
	$\Lambda_1-k_2\Lambda_2$ with $1<k_2\leq a_2-1 $,
	then $\CB(\lambda)$ is connected.
\end{theorem}
Let $\lambda$ be as in Theorem \ref{thm.mainbint}.
By \cite[Theorem 4.1]{H} \resp{Theorem \ref{thm.mainbint}},
we have $\B(\lambda)=\B_0(\lambda)$ \resp{$\CB(\lambda)=\CB_0(\lambda)$}.
Therefore, by Theorem \ref{thm.mainb0int},
we obtain the following corollary.

\begin{corollary}[$=$ Corollary \ref{cor.main}]
If $k_1=1$ or $k_2=1$, then
there exists an isomorphism $\B(\lambda) \rightarrow \CB(\lambda)$ of crystals that sends $\pi_\lambda$ to $u_\lambda$.
\end{corollary}

As an application of these results,
we give an algorithm for computing the number of elements of weight $\mu \in P$
in the crystal $\CB(\Lambda_1-\Lambda_2)$,
which is equal to the dimension of the weight space of weight $\mu$ in $V(\Lambda_1-\Lambda_2)$,
in the case that $A$ is symmetric, that is $a_1=a_2$.

This paper is organized as follows.
In Section \ref{sec.review}, we fix our notation, and recall some basic facts about extremal weight modules and their crystal bases.
Also, we recall the definition of  LS paths and the polyhedral realizations of $\CB(\pm\infty)$.
In Section \ref{sec.maintheorems}, we state our main theorems.
In Section \ref{sec.rank2}, we recall some properties of LS paths and the polyhedral realizations in the rank $2$ case.
Then we prove Theorems \ref{thm.mainb0int} and \ref{thm.mainbint} in Subsections \ref{sec.mainprf1} and \ref{sec.mainprf2}, respectively.
In Section \ref{sec.comp}, we give an algorithm for computing the number of elements of weight $\mu \in P$
in the  crystal $\CB(\Lambda_1-\Lambda_2)$.

\section{Review.}\label{sec.review}
\subsection{Kac-Moody algebras.}
Let $A$ be a  generalized Cartan matrix
and  $\mathfrak{g}= \mathfrak{g}(A)$ the Kac-Moody algebra
associated to $A$ over $\C$.
We denote by $\mathfrak{h}$ the Cartan subalgebra
of $\mathfrak{g}$,
$\{ \alpha_{i} \}_{i \in I} \subset \mathfrak{h^\ast}$
the set of simple roots,
and $\{ \alpha_i^\vee \}_{i \in I} \subset \mathfrak{h}$
the set of simple coroots,
where $I$  is the  index set.
Let  $s_i$ be the simple reflection with respect to  $\alpha_i$
for $i \in I$,  and let
$W=\langle s_i \mid i \in I \rangle$
be the Weyl group of $\mathfrak{g}$.
Let $\Delta_{\mathrm{re}}^+$ denote the set of positive real roots.
For a positive real root $\beta \in \Delta_{\mathrm{re}}^+$,
we denote by $\beta^{\vee}$ the dual root of $\beta$,
and by $s_\beta \in W $ the reflection with respect to $\beta$.
Let
$\{ \Lambda_i \}_{i \in I} \subset \mathfrak{h^\ast}$ be the fundamental weights for $\mathfrak{g}$,
i.e., $\langle \Lambda_i , \alpha_j^\vee \rangle=\delta_{i, j}$ for $i, j \in I$,
where $\langle \cdot , \cdot \rangle : \mathfrak{h^\ast} \times \mathfrak{h
}\rightarrow \C$
is the canonical pairing of $\mathfrak{h^\ast}$ and $\mathfrak{h}$.
We take an integral weight lattice $P$ containing $\al_i$ and $\Lambda_i$ for all  $i \in I$.
We denote by $P^+$ \resp{$-P^+$} the set of dominant
\resp{antidominant} integral weights.

Let $U_q(\Fg)$ be the quantized universal enveloping algebra over $\C(q)$
associated to $\Fg$,
and let $U_q^+(\Fg)$ \resp{$U_q^-(\Fg)$} be the positive \resp{negative}
part of $U_q(\Fg)$, that is, $\C(q)$-subalgebra generated by the Chevalley generators $E_i$ \resp{$F_i$}
of $U_q(\Fg)$ corresponding to the positive \resp{negative} simple root $\al_i$
\resp{$-\al_i$} for $i\in I$.

\subsection{Crystal bases and crystals.}\label{sec.crystal}
For details on crystal bases and crystals,
we refer the reader to \cite{Kocb} and \cite{HK}.
Let $\CB(\infty)$ \resp{$\CB(-\infty)$} be
the crystal basis of  $U_q^-(\Fg)$
(resp., $U_q^+(\Fg)$),
and let $u_\infty \in \CB(\infty)$ (resp., $u_{-\infty} \in \CB(-\infty)$)
be the element corresponding to $1 \in U_q^-(\Fg)$ \resp{$1 \in U_q^+(\Fg)$}.
Denote by $\ast: \CB(\pm \infty) \rightarrow \CB(\pm \infty)$ the $\ast$-operation on $\CB(\pm \infty)$;
see \cite[Theorem 2.1.1]{Kdem} and \cite[\S 8.3]{Kocb}.
For $\mu \in P$, let $\CT_\mu=\{ t_\mu\}$ be the crystal consisting of a single element $t_\mu$
such that
		\begin{align}
			\wt(t_\mu)=\mu, \quad \e_i t_\mu =\f_i t_\mu= \0, \quad
			\ep_i(t_\mu)=\ph_i(t_\mu)=-\infty \ \text{ for } i\in I,
		\end{align}
where $\0$ is an extra element not contained in any crystal.


Let $B$ be a normal crystal in the sense of \cite[\S 1.5]{K}.
We know from \cite[\S 7]{K} (see also \cite[Theorem 11.1]{Kocb}) that  $B$ has an action of the Weyl group $W$ as follows.
For $i\in I $ and $b \in B$, we set
\begin{align}
	S_ib:=
	\begin{cases}
		\f_i^{\pair{\wt(b)}{\alc_i}}b & \text{ if } \pair{\wt(b)}{\alc_i}\geq 0,\\
		\e_i^{-\pair{\wt(b)}{\alc_i}}b & \text{ if } \pair{\wt(b)}{\alc_i}\leq 0.\\
	\end{cases}
\end{align}
Then, for $w \in W$, we set $S_w:=S_{i_1} \cdots S_{i_k}$ if $w=s_{i_1} \cdots s_{i_k}$.
Notice that $\wt(S_w b)$= $w\wt(b)$ for $w\in W$ and $b\in B$.

\begin{definition}
	An element of a normal crystal $B$ is said to be extremal if for each $w\in W$ and $i\in I$,
	\begin{align}
		\e_i(S_w b)=\0 & \text{ if } \pair{\wt(S_w b)}{\alc_i}\geq 0,\\
		 \f_i(S_w b)=\0 & \text{ if }  \pair{\wt(S_w b)}{\alc_i}\leq 0.
	\end{align}
\end{definition}

\subsection{Crystal bases of extremal weight modules.}

Let $\mu \in P$ be an arbitrary integral weight. The extremal weight module $V(\mu)$ of extremal
weight $\mu$ is, by definition, the integrable $U_q(\Fg)$-module generated by a single element
 $v_\mu$ with the defining relation that $v_\mu$ is an extremal weight vector of weight $\mu$ in the sense
of \cite[Definition 8.1.1]{K}. We know from \cite[Proposition 8.2.2]{K} that $V(\mu)$  has a crystal basis $\CB(\mu)$.
Let $u_\mu$ denote the element of $\CB(\mu)$ corresponding to $v_\mu$.

\begin{remark}\label{rem.extmd}
We see from \cite[Proposition 8.2.2 (iv) and (v)]{K} that
$V(\mu)\cong V(w\mu)$ as $U_q(\Fg)$-modules,
and $\CB(\mu)\cong \CB(w\mu)$ as crystals for all $\mu \in P$ and $w\in W$.
Also, we know from the comment at the end of \cite[\S8.2]{K} that
if $\mu \in P^+$ \resp{$\mu \in -P^+$},
then $V(\mu)$ is isomorphic, as a $U_q(\Fg)$-module, to the integrable highest
\resp{lowest}
weight
module of highest \resp{lowest}
weight $\mu$, and $\CB(\mu)$ is isomorphic, as a crystal, to its
crystal basis.
So, we focus on those $\mu \in P$ satisfying the condition that
\begin{align}
	W\mu \cap (P^+ \cup -P^+)=\emptyset\label{eq.Are}
\end{align}
\end{remark}

The crystal basis $\CB(\mu)$ of $V (\mu)$ can be realized (as a crystal) as follows. We set
\begin{align}
	\CB :=\bigsqcup_{\mu\in P} \CB(\infty) \otimes \CT_\mu \otimes \CB(-\infty);
\end{align}
in fact, $\CB$ is isomorphic, as a crystal, to the crystal basis $\CB(\tilde{U}_q(\Fg))$ of the modified quantized universal enveloping algebra $\tilde{U}_q(\Fg)$ associated to $\Fg$
(see \cite[Theorem 3.1.1]{K}). Denote by
$\ast : \CB \rightarrow \CB$ the $\ast$-operation on $\CB$
(see \cite[Theorem 4.3.2]{K});
we know from \cite[Corollary 4.3.3]{K} that for
$b_1 \in \CB(\infty)$, $b_2 \in \CB(-\infty)$, and $\mu \in P$,
\begin{align}\label{eq.starope}
	(b_1 \otimes t_\mu \otimes b_2)^\ast = b_1^\ast \otimes t_{-\mu -\wt(b_1)- \wt(b_2)} \otimes b_2^\ast.
\end{align}

\begin{remark}
	The weight of $(b_1 \otimes t_\mu \otimes b_2)^\ast$ is equal to $-\mu$ for all $b_1 \in \CB(\infty)$ and $b_2 \in \CB(-\infty)$ since $\wt(b_1^\ast) = \wt(b_1)$ and $\wt(b_2^\ast) = \wt(b_2)$.
\end{remark}

Because $\CB$ is a normal crystal by \cite[§2.1 and Theorem 3.1.1]{K}, $\CB$ has the action of
the Weyl group $W$ (see \S \ref{sec.crystal}). We know the following proposition from
\cite[Proposition 8.2.2 (and Theorem 3.1.1)]{K}.

\begin{theorem}\label{thm.cbext}
	For $\mu \in P$, the subset
	\begin{align}\label{eq.cb}
		\{b\in \CB(\infty) \otimes \CT_\mu \otimes \CB(-\infty) \mid b^\ast
		\text{\rm{ is extremal}} \}
	\end{align}
	is a subcrystal of $ \CB(\infty) \otimes \CT_\mu \otimes \CB(-\infty)$, and is isomorphic,
	as a crystal, to the crystal basis $\CB(\mu)$ of the extremal weight module $V(\mu)$ of extremal weight $\mu$.
In particular, $u_{\infty} \otimes t_\mu \otimes u_{-\infty} \in \CB(\infty) \otimes \CT_\mu \otimes \CB(-\infty)$ is contained in the set above, and corresponds to $u_\mu \in \CB(\mu)$ under the isomorphism.
\end{theorem}

\subsection{Lakshmibai-Seshadri paths.}\label{section.LSpath}
We recall Lakshmibai-Seshadri paths from
\cite[\S 2, \S 4]{L}.
In this subsection,
we fix an integral weight  $\mu \in P$.

\begin{definition}\label{order}
For $\nu, \nu' \in W\mu$,
we write $\nu \geq \nu' $
if there exist a sequence
$\nu=\nu_0, \nu_1, \ldots$, $\nu_u=\nu'$
of elements in
$W\mu$
and a sequence $\beta_1, \beta_2, \ldots , \beta_u$
of positive real roots such that
$\nu_k=s_{\beta_k}(\nu_{k-1})$ and
$\langle \nu_{k-1}, \beta^{\vee}_k \rangle < 0$
for each $k=1,2, \ldots, u$.
If $\nu \geq \nu'$, then we define $\mathrm{dist}(\nu, \nu')$
to be the maximal length $u$ of all possible such sequences
$\nu=\nu_0, \nu_1, \ldots, \nu_u=\nu'$.
\end{definition}

\begin{remark}\label{rem.1}
For $\nu, \nu' \in W\mu$ such that
$\nu > \nu'$ and $\mathrm{dist}(\nu, \nu')=1$,
there exists a unique positive real root $\beta \in \derep$ such that
$\nu'=s_{\beta}(\nu)$.
\end{remark}

The Hasse diagram of $W\mu$ is, by definition, the $\derep$-labeled,
directed graph with vertex set $W\mu$,
and edges of the following form:
$\nu \xleftarrow{\beta} \nu'$
for $\nu, \nu' \in W\mu$ and
$\beta \in \derep$
such that $\nu >\nu' $ with $\mathrm{dist}(\nu, \nu')=1$ and $\nu'=s_{\beta}(\nu)$.

\begin{definition}
Let $\nu, \nu' \in W\mu$ with $\nu > \nu' $,
and let $0<\sigma<1$ be a rational number.
A $\sigma$-chain for $(\nu, \nu')$ is a sequence
$\nu = \nu_0, \ldots, \nu_u=\nu'$
of elements of $W\mu $ such that
$\mathrm{dist}(\nu_{k-1}, \nu_k)=1$
and $ \sigma \langle \nu_{k-1},
\beta_k^\vee \rangle \in \mathbb{Z}_{<0}$
for all $k=1,2,\ldots, u$,
where $\beta_k$ is the unique positive real root satisfying
$\nu_k=s_{\beta_k}(\nu_{k-1})$.
\end{definition}

\begin{definition}
Let $\nu_1> \cdots> \nu_u$ be a finite sequence of elements in $W\mu$,
and let $0=\sigma_0<\cdots<\sigma_u=1$ be a finite sequence of rational numbers.
The pair
$\pi=(\nu_1,\ldots,\nu_u; \sigma_0, \ldots, \sigma_u)$
is called a Lakshmibai-Seshadri (LS for short) path of shape $\nu$
if there exists
a $\sigma_k$-chain for $(\nu_k,\nu_{k+1})$
for each $k=1, \ldots, u-1$.
We denote by $\B(\mu)$
the set of LS paths of shape $\mu$.
\end{definition}

Let  $[0,1]:= \{ t\in \R \mid 0 \leq t \leq 1 \}$.
We identify
$\pi=(\nu_1,\ldots,\nu_u;\sigma_0, \ldots , \sigma_u) \in \B(\mu)$ with the following piecewise-linear continuous map
$\pi:[0,1]\to \R \otimes_{\Z} P $:
\begin{align}
	\pi(t)=\sum^{j-1}_{k=1}(\sigma_k-\sigma_{k-1})\nu_k+
	(t-\sigma_{j-1})\nu_j \quad \text{ for }
	\sigma_{j-1}\le t\le\sigma_j, \  1\le j\le u.
\end{align}

We endow $\B(\mu)$ with
a crystal structure as follows.
First, we define $\mathrm{wt}(\pi):=\pi(1)$
for $\pi \in \B(\mu) $;
we know from \cite[Lemma 4.5 (a)]{L} that $\pi(1)\in P$.
Next, for $\pi \in \B(\mu) $ and $i \in I$,
\begin{gather}
	H^\pi_i(t):=\langle\pi(t),\alpha_i^{\vee}\rangle
	\quad \text{ for }  0 \leq t \leq1,\\
	m^\pi_i:=\min\{H^\pi_i(t) \mid 0\leq t \leq 1\}.
\end{gather}
From \cite[Lemma 4.5 (d)]{L}, we know that
\begin{align}\label{eq.int}
	\text{all local minimum values of }  H^\pi_i(t)
	\text{ are integers};
\end{align}
in particular, $m_i^\pi \in \mathbb{Z}_{\leq 0}$
and $H^\pi_{i}(1)-m_i^\pi \in \mathbb{Z}_{\geq 0}$.
We define $\e_i \pi$ as follows.
If $m^\pi_i=0$, then we set $\e_i\pi:=\mathbf{0}$.
If $m^\pi_i\le-1$, then we set
\begin{align}
t_1&:=\min\{t\in[0,1] \mid H^\pi_i(t)=m^\pi_i\}, \label{eq.et1} \\
t_0&:=\max\{t\in[0,t_1] \mid H^\pi_i(t)=m^\pi_i+1\}; \label{eq.et0}
\end{align}
we see by (\ref{eq.int}) that
\begin{align}\label{eq.estrict}
	H^\pi_i(t)  \text{ is strictly decreasing on } [t_0, t_1].
\end{align}
We define
\begin{align}
	(\e_i\pi)(t):=
	\begin{cases}
		\pi(t)
			& \text{if} \ 0\le t\le t_0, \\
		s_i(\pi(t)-\pi(t_0))+\pi(t_0)
			& \text{if} \ t_0\le t\le t_1,\\
		\pi(t)+\alpha_i
			& \text{if} \ t_1\le t \leq 1;
	\end{cases}
\end{align}
we know from
\cite[$\S 4$]{L} that $\e_i \pi \in \B(\mu)$.
Similarly, we define
$ \f_i \pi$ as follows.
If $H^\pi_i(1)-m_i^\pi=0$, then we set $\f_i\pi:=\mathbf{0}$.
If $H^\pi_i(1)-m_i^\pi \geq 1$, then  we set
\begin{align}
	t_0 &:=\max\{t\in[0,1]
		\mid 	H^\pi_i(t)=m^\pi_i\}, \label{eq.ft0} \\
	t_1&:=\min\{t\in[t_0,1] \mid 		H^\pi_i(t)=m^\pi_i+1\}; \label{eq.ft1}
\end{align}
we see {by} (\ref{eq.int}) that
	$H^\pi_i(t)$   is strictly increasing on  $[t_0, t_1]$.
We define
\begin{align}
	(\f_i\pi)(t):=
	\begin{cases}
		\pi(t)
			& \text{if }  0\leq t\leq t_0 , \\
		s_i(\pi(t)-\pi(t_0))+\pi(t_0)
			& \text{if }  t_0 \leq t\leq t_1, \\
		\pi(t)-\alpha_i
			& \text{if } t_1\leq t \leq1;
	\end{cases}
\end{align}
we know from \cite[$\S 4$]{L}
that $f_i \pi \in \B(\mu)$.
We set $\e_i\mathbf{0}= \f_i\mathbf{0} := \mathbf{0}$ for
$i \in I $.
Finally,
for $\pi \in \B(\mu)$ and $i \in I$,
we set
\begin{align}
	\ep_i (\pi)
	:=\max \{k\in\Z_{\geq 0}
	\mid \e^k_i \pi \neq \0\}, \quad
	\ph_i (\pi) :=\max \{k\in\Z_{\geq 0}
	\mid \f^k_i \pi \neq \0\}.
\end{align}
We know from {\cite[Lemma 2.1 (c)]{L}} that
\begin{align}\label{eq.LSepph}
	 \ep_i(\pi)=-m^\pi_i, \quad
	 \ph_i (\pi)=H^\pi_i(1)-m_i^\pi.
\end{align}
\begin{theorem}[{\cite[\S 2, \S 4]{L}}]
The set $\B(\mu)$,
together with the maps
$\wt: \B(\mu)  \to P$,
$\e_i, \f_i: \B(\mu) \to
	\B(\mu) \cup \{ \0 \}$, $i\in I$,
and $\ep_i, \ph_i : \B(\mu) \to
	 \Z_{\geq 0}$, $i \in I$, is a crystal.
\end{theorem}

\subsection{Polyhedral realization of $\mathcal{B}(\pm\infty)$ and $\CB(\infty) \otimes \CT_\mu \otimes \CB(-\infty)$.}\label{sec.poly}
Let us recall the polyhedral realization of $\mathcal{B}(\infty)$
and $\mathcal{B}(-\infty)$ from \cite{NZ}.
We fix an infinite sequence
$\iota^+=(\ldots, i_k, \ldots, i_2, i_1)$
of elements of  $I$ such that $i_k \neq i_{k+1} $ for $k \in \Z_{\geq 1}$, and
$	\# \{ k \in \Z_{\geq 1} \mid i_k=i \}=\infty$ for each $i \in I$.
Similarly, we fix an infinite sequence
$\iota^-=(i_0, i_{-1}, \ldots, i_k, \ldots)$
of elements of  $I$ such that $i_k \neq i_{k-1} $ for $k \in \Z_{\leq 0}$, and
$	\# \{ k \in \Z_{\leq 0} \mid i_k=i \}=\infty $ for each $i \in I$.
We set
\begin{align}
	\Z^{+\infty}_{\geq 0}
	&:=\{ (\ldots, y_k, \ldots, y_2, y_1)
	\mid y_k \in \Z_{\geq 0} \text{ and } y_k=0 \text{ for } k \gg 0 \},\\
	\Z^{-\infty}_{\leq 0}
	&:=\{ (y_0, y_{-1}, \ldots, y_k, \ldots)
	\mid y_k \in \Z_{\leq 0} \text{ and } y_k=0 \text{ for } k \ll 0  \}.
\end{align}
We endow $\Z^{+\infty}_{\geq 0}$ and $\Z^{-\infty}_{\leq 0}$ with crystal structures as follows.
Let $y^+=(\ldots, y_k, \ldots, y_2, y_1) \in \Z^{+\infty}_{\geq 0}$
and $y^-=(y_0, y_{-1} \ldots, y_k \ldots) \in \Z^{-\infty}_{\leq 0}$.
For $k \geq 1$, we set
\begin{align}
	\sigma^+_k(y^+) = y_k + \sum_{j>k}\pair{\al_{i_j}}{\alc_{i_k}}y_j,
\end{align}
and for $k\leq 0$, we set
\begin{align}
		\sigma^-_k(y^-) =-y_k - \sum_{j<k}\pair{\al_{i_j}}{\alc_{i_k}}y_j;
\end{align}
since $y_j=0$ for $|j|\gg 0$, we see that  $\sigma^\pm_k(y)$  is well-defined, and $\sigma^\pm_k(y) =0$ for $|k| \gg 0$.
For $i \in I$,  we set
$\sigma^+_{(i)}({y^+}):= \max \{ \sigma^+_k(y^+)\mid k \geq 1,   i_k =i \}$
and $\sigma^-_{(i)}({y^-}):= \max \{ \sigma^-_k(y^-)\mid k \leq 0,   i_k =i \}$,
and define
\begin{align}
	&M^+_{(i)}= M^+_{(i)}(y^+):=
	\{ k \mid k \geq 1, i_k=i, \sigma^+_k(y^+) =\sigma^+_{(i)}(y^+) \},\\
	&M^-_{(i)}= M^-_{(i)}(y^-):=
	\{ k \mid k \leq 0, i_k=i, \sigma^-_k(y^-) =\sigma^-_{(i)}(y^-) \}.
\end{align}
Note that $\sigma^\pm_{(i)}(y^\pm) \geq 0$, and that $M^\pm_{(i)}= M^\pm_{(i)}(y^\pm)$ is a finite set
if and only if $\sigma^\pm_{(i)}(y^\pm) > 0$.
We define the maps
$\e_i, \f_i: \Z^{+\infty} \rightarrow \Z^{+\infty} \sqcup \{ \0 \}$
and  $\e_i, \f_i: \Z^{-\infty} \rightarrow \Z^{-\infty} \sqcup \{ \0 \}$
by
\begin{align}
	\e_i y^+ &:=
	\begin{cases}
		(\ldots, y'_k, \ldots, y'_2, y'_1)
		\text{ with } y'_k:=y_k - \delta_{k, \max{M^+_{(i)}}}
		& \text{ if } \sigma^+_{(i)}({y^+}) > 0,\\
		\0 & \text{ if } \sigma^+_{(i)}({y^+}) = 0,
	\end{cases}\\
	\f_i y^+ &:=(\ldots, y'_k, \ldots, y'_2, y'_1)
	\text{ with } y'_k:=y_k + \delta_{k, \min{M^+_{(i)}}},\\
	\e_i y^- &:= (y'_0, y'_{-1} \ldots, y'_k \ldots)
	\text{ with } y'_k:=y_k - \delta_{k, \max{M^-_{(i)}}},\\
		\f_i y^- &:=
			 \begin{cases}
					(y'_0, y'_{-1} \ldots, y'_k \ldots)
					\text{ with } y'_k:=y_k + \delta_{k, \min{M^-_{(i)}}} & \text{ if } \sigma^-_{(i)}({y^-}) > 0,\\
				 \0 & \text{ if } \sigma^-_{(i)}({y^-}) = 0,
			 \end{cases}
\end{align}
respectively.
Moreover, we define
\begin{gather}
	\wt(y^+):=-\sum_{j\geq 1} y_j \al_{i_j}, \quad
	\ep_i(y^+):=\sigma^+_{(i)}(y^+), \quad
		 \ph_i(y^+):=\ep_i(y^+)+\pair{\wt(y^+)}{\alc_i},\\
	\wt(y^-):=-\sum_{j\leq 0} y_j \al_{i_j}, \quad
	\ph_i(y^-):=\sigma^-_{(i)}(y^-), \quad
	\ep_i(y^-):=\ph_i(y^-)-\pair{\wt(y^-)}{\alc_i}.
\end{gather}
These maps make $\Z_{\geq 0}^{+\infty}$ \resp{$\Z_{\leq 0}^{-\infty}$} into a crystal for $\Fg$;
we denote this crystal by $\Z^{+\infty}_{\iota^+}$ \resp{$\Z^{-\infty}_{\iota^-}$}.
\begin{theorem}[{\cite[Theorem 2.5]{NZ}}]\label{thm.polyh}
There exists an embedding
$\Psi^+_{\iota^+}: \CB(\infty) \hookrightarrow \Z^{+\infty}_{\iota^+}$
of crystals which sends $u_\infty \in \CB(\infty )$
to $(\ldots, 0 \ldots, 0, 0) \in \Z_{\iota^+}^{+\infty}$.
Similarly, there exists an embedding
$\Psi^-_{\iota^-}: \CB(-\infty) \hookrightarrow \Z^{-\infty}_{\iota^-}$
of crystals which sends $u_{-\infty} \in \CB(-\infty) $
to $(0, 0, \ldots, 0, \ldots) \in \Z^{-\infty}_{\iota^-}$.
\end{theorem}
The next corollary follows immediately from Theorem \ref{thm.polyh}.
\begin{corollary}\label{cor.polyhofex}
	For each $\mu \in P$,
	there exists an embedding
	$\CB(\infty) \otimes \CT_\mu \otimes  \CB(-\infty) \hookrightarrow
	\Z^{+\infty}_{\iota^+} \otimes \CT_\mu \otimes\Z^{-\infty}_{\iota^-}$
	 of crystals
	which sends
	$u_\infty  \otimes t_\mu \otimes u_{-\infty} \in \CB(\infty) \otimes \CT_\mu \otimes  \CB(-\infty)$  to $(\ldots, 0, \ldots, 0,0) \otimes t_\mu \otimes  (0, 0, \ldots, 0, \ldots) \in \Z^{+\infty}_{\iota^+} \otimes \CT_\mu \otimes\Z^{-\infty}_{\iota^-}$.
\end{corollary}
We define $\ast$-operations on $\img(\Psi_{\iota^\pm}^{\pm})$ and
$\img(\Psi_{\iota^+}^{+}) \otimes \CT_\mu \otimes \img(\Psi_{\iota^-}^{-})$
by the following commutative diagrams, respectively:
\begin{align}
	\begin{CD}
    \CB(\pm\infty) @>{\ast}>> \CB(\pm\infty) \\
  	@V{\Psi_{\iota^\pm}^{\pm}}VV    @VV{\Psi_{\iota^\pm}^{\pm}}V \\
    \img(\Psi_{\iota^\pm}^{\pm})  @>{\ast}>>  \img(\Psi_{\iota^\pm}^{\pm}),
  \end{CD}
\end{align}
\begin{align}
	\begin{CD}
    \CB(\infty) \otimes \CT_\mu \otimes \CB(-\infty)  @>{\ast}>> \CB(\infty) \otimes \CT_\mu \otimes \CB(-\infty) \\
  	@V{\Psi_{\iota^+}^{+} \otimes \, \mathrm{id} \,  \otimes \Psi_{\iota^-}^{-} }VV    @VV{\Psi_{\iota^+}^{+} \otimes \, \mathrm{id} \,  \otimes \Psi_{\iota^-}^{-}}V \\
    \img(\Psi_{\iota^+}^{+}) \otimes \CT_\mu \otimes\img(\Psi_{\iota^-}^{-})   @>{\ast}>>  \img(\Psi_{\iota^+}^{+}) \otimes \CT_\mu \otimes\img(\Psi_{\iota^-}^{-}).
  \end{CD}
\end{align}
Then we deduce from \cite[Remark in \S 2.4]{NZ} that
if $z_1=(\ldots, c_2, c_1) \in \img(\Psi_{\iota^+}^{+})$,
then $z_1^{\ast}=\f_{i_1}^{c_1} \f_{i_2}^{c_2} \cdots (\ldots, 0, 0)$.
Similarly, we see that
if  $z_2=(c_0, c_{-1}, \ldots ) \in \img(\Psi_{\iota^-}^{-})$,
then
$z_2^{\ast}=\e_{i_0}^{-c_0} \e_{i_{-1}}^{-c_{-1}} \cdots (0, 0, \ldots) $.
Moreover, we see  by \eqref{eq.starope} that
if $z_1 \in \img(\Psi_{\iota^+}^{+})$ and $z_2 \in \img(\Psi_{\iota^-}^{-})$,
then
\begin{align}
	(z_1 \otimes t_\mu \otimes z_2)^\ast = z_1^\ast \otimes t_{-\mu -\wt(z_1)- \wt(z_2)} \otimes z_2^\ast.
\end{align}


By the tensor product rule of crystals, we can describe the crystal structure of
$ \Z^{+\infty}_{\iota^+} \otimes \CT_\mu \otimes\Z^{-\infty}_{\iota^-}$ as follows.
Let $y=y^+ \otimes t_\mu \otimes y^-$
with  $y^+=(\ldots, y_2, y_1) \in \Z^{+\infty}_{\iota^+} $
and $y^-= ( y_0, y_{-1}, \ldots ) \in \Z^{-\infty}_{\iota^-}$.
We set
\begin{align}
	\sigma_k(y):=
	\begin{cases}
		\sigma_k^+(y^+) & \text{ if } k\geq 1,\\
		\sigma_k^-(y^-) -\pair{\wt(y)}{\alc_{i_k}} & \text{ if } k\leq 0.\\
	\end{cases}
\end{align}
For $i \in I$, we set
$\sigma_{(i)}({y}):= \max\{\sigma_k(y)\mid k\in\Z,  i_k =i\}$ and
\begin{align}\label{eq.mi}
	M_{(i)}= M_{(i)}(y):=
	\{ k \mid i_k=i, \sigma_k(y) =\sigma_{(i)}(y) \}.
\end{align}
Then we see that
\begin{align}
	\wt(y) =\mu-\sum_{j\in \Z} y_j \al_{i_j}, \quad \ep_i(y) = \sigma_{(i)}(y),
	\quad \ph_i(y) =\ep_i(y)+ \pair{\wt(y)}{\alc_i},
\end{align}
\begin{align}
	&\e_i y =
	\begin{cases}
		(\ldots, y'_2, y'_1)\otimes  t_\mu \otimes ( y'_0, y'_{-1}, \ldots )
		\text{ with } y'_k:=y_k - \delta_{k, \max{M_{(i)}}} & \text{ if } \ep_i(y) > 0,\\
		\0 & \text{ if } \ep_i(y) =0,
	\end{cases}\\
	&\f_i y =
	\begin{cases}
		(\ldots, y'_2, y'_1)\otimes  t_\mu \otimes ( y'_0, y'_{-1}, \ldots )
		\text{ with } y'_k:=y_k + \delta_{k, \min{M_{(i)}}} & \text{ if } \ph_i(y) > 0,\\
		\0 & \text{ if } \ph_i(y) =0.
	\end{cases}
\end{align}

\section{Main results.}\label{sec.maintheorems}
In the following, we assume that
the generalized Cartan matrix $A$ is  of the form
\begin{align}\label{eq.gcm}
	A=
	\begin{pmatrix}
		2 & -a_1 \\-a_2 & 2\\
	\end{pmatrix}, \
	\text{ where }
	a_1, a_2 \in \Z_{\geq 2}
	\text{ with }
	a_1 a_2>4;
\end{align}
note that $\al_1=2\Lambda_1-a_2\Lambda_2$ and  $\al_2=-a_1\Lambda_1+2\Lambda_2$.
We set $P=\Z \Lambda_1 \oplus \Z \Lambda _2$.
Let $\lambda \in P$ be an integral weight of the form either
\eqref{enu.1} or \eqref{enu.2}:
\begin{enumerate}[\upshape(i)]
	\item $\lambda=k_1\Lambda_1-k_2\Lambda_2$ for some
	$k_1, k_2 \in \Z_{>0}$ such that $k_2\leq k_1<(a_1-1)k_2$;\label{enu.1}
	\item $\lambda=k_1\Lambda_1-k_2\Lambda_2$ for some
	$k_1, k_2 \in \Z_{>0}$ such that $k_1<k_2\leq (a_2-1)k_1 $.\label{enu.2}
\end{enumerate}
\begin{remark}
	Let $\mathbb{O}:=\{ W\mu \mid \mu \in P \}$ be
	the set of $W$-orbits in $P$.
	We know from \cite[Theorem 3.1]{H} that
	$O\in \mathbb{O}$ satisfies condition \eqref{eq.Are},
	that is,
	$O  \cap (P^+ \cup -P^+) = \emptyset$  if and only if
	$O$ contains
	an integral weight 
	of the form either \eqref{enu.1} or \eqref{enu.2} above.
\end{remark}
%
%
Let  $\B_0(\lambda)$ \resp{$\CB_0(\lambda)$} be the connected component of $\B(\lambda)$ \resp{$\CB(\lambda)$} containing $\pi_\lambda:=(\lambda; 0,1 )$ \resp{$u_\lambda$}.
\begin{theorem}[will be proved in \S \ref{sec.mainprf1}]\label{thm.mainb0}
	Let $\lambda$ be an integral weight of the form either \eqref{enu.1} or \eqref{enu.2} above.
	There exists an isomorphism $\B_0(\lambda) \rightarrow \CB_0(\lambda)$ of crystals that sends $\pi_\lambda$ to $u_\lambda$.
\end{theorem}
\begin{theorem}[will be proved in \S \ref{sec.mainprf2}]\label{thm.mainb}
	Assume that $k_1=1$ or $k_2=1$, that is,
	$\lambda \in P$ is of the form either
	$k_1\Lambda_1-\Lambda_2$ with $1\leq k_1 < a_1-1$ or
	$\Lambda_1-k_2\Lambda_2$ with $1<k_2\leq a_2-1 $.
	For $b \in \CB(\lambda)$, there exist $i_1, \ldots , i_r \in I$
	such that $b=\f_{i_r} \cdots \f_{i_1}u_\lambda$ or $b=\e_{i_r} \cdots \e_{i_1}u_\lambda$.
	In particular, the crystal graph of $\CB(\lambda)$ is connected.
\end{theorem}
Let $\lambda$ be as in Theorem \ref{thm.mainb}.
By \cite[Theorem 4.1]{H} \resp{Theorem \ref{thm.mainb}},
we have $\B(\lambda)=\B_0(\lambda)$ \resp{$\CB(\lambda)=\CB_0(\lambda)$}.
Therefore, by Theorem \ref{thm.mainb0},
we obtain the following corollary.

\begin{corollary}\label{cor.main}
If $k_1=1$ or $k_2=1$, then
there exists an isomorphism $\B(\lambda) \rightarrow \CB(\lambda)$ of crystals that sends $\pi_\lambda$ to $u_\lambda$.
\end{corollary}

\section{Rank $2$ case.}\label{sec.rank2}
\subsection{LS paths in the rank $2$ case.}
%
Let $A$ and  $\lambda=k_1\Lambda_1-k_2\Lambda_2$ be as in \S \ref{sec.maintheorems}.
In this subsection,
we recall some properties of $\B(\lambda)$  from \cite{H}.
We define the sequence $\{ p_m \}_{m\in \Z} $ of integers
by the following recursive formulas: for $m\geq 0$,
\begin{align}\label{eq.pm}
	p_0:=k_2, \quad
	p_1:=k_1, \quad
	p_{m+2}:=
	\begin{cases}
		a_2 p_{m+1}-p_m & \text{if} \ m \ \text{is even}, \\
		a_1 p_{m+1}-p_m & \text{if} \ m \ \text{is odd}; \\
	\end{cases}
\end{align}
for  $m<0$,
	\begin{align}\label{eq.pm2}
		p_{m}=
		\begin{cases}
			a_2 p_{m+1}-p_{m+2} & \text{if} \ m \ \text{is even}, \\
			a_1 p_{m+1}-p_{m+2} & \text{if} \ m \ \text{is odd}; \\
		\end{cases}
\end{align}
it follows from \cite[Remark 3.7]{H}
(and  the comment in \cite[\S 3.1]{Yu})
that
\begin{align}\label{eq.pmpos}
	p_m>0 \text{ for all } m \in \Z.
\end{align}
%
Notice that $W=\{ x_m \mid m \in \Z \}$, where
\begin{align}
x_{m}:=
	\begin{cases}
		(s_2 s_1)^n & \text{ if } m=2n \text{ with } n \in \Z_{\geq 0}, \\
		s_1(s_2 s_1)^n & \text{ if } m=2n+1 \text{ with } n \in \Z_{\geq 0}, \\
		(s_1 s_2)^{-n} & \text{ if } m=2n \text{ with } n \in \Z_{\leq 0}, \\
		s_2(s_1 s_2)^{-n} & \text{ if } m=2n-1 \text{ with } n \in \Z_{\leq 0}.
	\end{cases}
\end{align}
Then we have
\begin{align}\label{eq.xm}
	x_m\lambda=
	\begin{cases}
		p_{m+1}\Lambda_1-p_m\Lambda_2 &
		\text{if } m  \text{ is even},\\
		-p_m\Lambda_1+p_{m+1}\Lambda_2 &
		\text{if } m  \text{ is odd},\\
	\end{cases}
\end{align}
for $m \in \Z$
by \cite[Lemma 3.3]{H}.
\begin{proposition}[{\cite[Proposition 3.8]{H}}]\label{prop.Hasse}
	The Hasse diagram of $W\lambda$ in the ordering of {\rm Definition \ref{order}} is
	\begin{align}\label{eq.Hasse}
		\cdots \xleftarrow{\alpha_1} x_2\lambda \xleftarrow{\alpha_2} x_1\lambda \xleftarrow{\alpha_1}
		x_0\lambda
		\xleftarrow{\alpha_2} x_{-1}\lambda \xleftarrow{\alpha_1} x_{-2}\lambda
		\xleftarrow{\alpha_2} \cdots.
	\end{align}
\end{proposition}
For each $\nu \in W\lambda $,
there exists a unique $m\in \Z$ such that $\nu=x_m\lambda$. Then
we define  $z(\nu):=m$.
We set
\begin{align}
	\B_1(\lambda):=\{ (\nu_1, \ldots, \nu_u;
	\sigma_1, \ldots, \sigma_u) \in \B(\lambda) \mid
	z(\nu_v)-z(\nu_{v+1})=1 \text{ for } v=1, \ldots, u-1 \};
\end{align}
note that $\pi_\lambda=(\lambda; 0, 1) \in \B_1(\lambda)$.
We know from \cite[Theorem 4.17]{H} that
$\B_1(\lambda) \cup \{ \0 \}$ is stable under the action of  $\e_i, \f_i$ for $i \in I=\{ 1, 2 \}$.
Hence, $\B_1(\lambda)$  is a subcrystal of $\B(\lambda)$
(but not necessarily, a connected component of $\B(\lambda)$).
Since $\B_0(\lambda)$ is the connected component of $\B(\lambda)$  containing $\pi_\lambda$,
by the definition,
it follows that
\begin{align}\label{eq.b01}
	\B_0(\lambda) \subset \B_1(\lambda).
\end{align}
We deduce by \eqref{eq.pmpos} and   \eqref{eq.xm}
that an element  $\pi\in \B_1(\lambda)$  is of the form
%
\begin{align}\label{eq.pib1}
	\pi =
	\biggl(
	x_{m}\lambda, x_{m-1}\lambda, \ldots, x_{n}\lambda;
	0, \dfrac{q_{m}}{p_{m}}, \dfrac{q_{m-1}}{p_{m-1}}, \ldots, \dfrac{q_{n+1}}{p_{n+1}}, 1
	\biggr),
\end{align}
where $n\leq m$, and $q_m, q_{m-1}, \ldots, q_{n+1}$ are integers satisfying
\begin{align}
	0<q_j<p_j \text{ for } n+1\leq j \leq m, \text{ and }
	\dpq{j+1}<\dpq{j}  \text{ for } n+1\leq j \leq m-1.\label{eq.prq}
\end{align}
%
\begin{remark}\label{rem.alt}
	Let
	$\pi = (x_{n+s-1}\lambda, \ldots, x_{n+1}\lambda, x_n\lambda;
	\sigma_0, \sigma_1, \ldots, \sigma_s ) \in \B_1(\lambda)$,
	and $i\in I=\{ 1, 2\}$.
	We see by \eqref{eq.pmpos},  \eqref{eq.xm}, and the definition of $\B_1(\lambda)$ that
	the function $H_i^\pi(t)=\pair{\pi(t)}{\alc_i}$ attains
	either a maximal value or a minimal value at $t \in [0, 1]$ if and only if
	$t \in \{ 0=\sigma_0, \sigma_1, \ldots, \sigma_s=1 \}$.
	Moreover,
	if $H_i^\pi(t)$ attains a minimal (resp., maximal) value at
	$t=\sigma_v$, then $H_i^\pi(t)$
	attains a minimal (resp., maximal) value at $t=\sigma_u$ for all  $u=0, 1, \ldots, s$
	such that $u  \equiv v$ mod $2$.
\end{remark}
\begin{remark}\label{rem.primeb1}
	Assume that  $k_1$ and $k_2$ are relatively prime.
	We know from \cite[Lemma 4.5 (3)]{H} that
	an LS path of shape $\lambda=k_1\Lambda_1-k_2\Lambda_2$ is of the form
	\eqref{eq.pib1}.
	Hence we have   $\B_1(\lambda)=\B(\lambda)$.
\end{remark}
\begin{theorem}[{\cite[Theorem 4.1]{H}}]\label{thm.LScon}
		If $k_1=1$ or $k_2=1$,
		that is,
		$\lambda \in P$ is of the form either
		$k_1\Lambda_1-\Lambda_2$ with $1\leq k_1 < a_1-1$ or
		$\Lambda_1-k_2\Lambda_2$ with $1<k_2\leq a_2-1 $,
		then the crystal graph of $\B(\lambda)$ is connected.
		Otherwise, the crystal graph of  $\B(\lambda)$ has infinitely many connected
		components.
\end{theorem}

\subsection{Polyhedral realizations of $\CB(\pm\infty)$ in the rank $2$ case.}\label{sec.poly2}
Let $A$ and  $\lambda=k_1\Lambda_1-k_2\Lambda_2$ be as in \S \ref{sec.maintheorems}.
Set $\iota^+:=(\ldots, 2, 1, 2, 1)$ and $\iota^-:=(2, 1, 2, 1, \ldots)$.
We define the sequence $\{ c_j \}_{j\in \Z} $ of integers
by the following recursive formulas: for $j\geq 1$,
\begin{align}
	c_1:=1, \quad
	c_2:=a_1, \quad
	c_{j+2}:=
	\begin{cases}
		a_1 c_{j+1}-c_j & \text{ if }  j  \text{ is even}, \\
		a_2 c_{j+1}-c_j & \text{ if }  j  \text{ is odd}; \\
	\end{cases}
\end{align}
for $j\leq 0$,
\begin{align}
	c_0:=1, \quad
	c_{-1}:=a_2, \quad
	c_{j-2}:=
	\begin{cases}
		a_1  c_{j-1}-c_j & \text{ if }  j  \text{ is even}, \\
		a_2  c_{j-1}-c_j & \text{ if }  j  \text{ is odd}. \\
	\end{cases}
\end{align}
Applying \cite[Theorem 4.1]{NZ} to our rank $2$ case,
we obtain the following explicit description of
the image of $\Psi^+_{\iota^+}$ and $\Psi^-_{\iota^-}$.
\begin{proposition}\label{prop.impsi}
	It hold that
	\begin{align}
		\img(\Psi^+_{\iota^+})&=
		\{(\ldots, y_2, y_1 )\in \Z^{+\infty}_{\geq 0} \mid c_l y_l - c_{l-1} y_{l+1} \geq 0 \text{\rm{ for }} l \geq 2 \},\\
		\img(\Psi^-_{\iota^-})&=
		\{(y_0, y_{-1}, \ldots )\in \Z^{-\infty}_{\leq 0} \mid c_l y_l - c_{l+1} y_{l-1} \leq 0 \text{\rm{ for }} l \leq -1 \}.
	\end{align}
\end{proposition}

The following  lemma will be needed in the next section.
Recall that $\{p_m\}_{m\in \Z}$ is defined by the recursive formulas \eqref{eq.pm} and \eqref{eq.pm2}.
\begin{lemma}\label{lem.wd}
	Let $m, n\in \Z$ be such that $n < m$
	 and  let $q_{n+1}, q_{n+2}, \ldots, q_{m} \in \Z$ be such that
	 $0<q_j<p_j$ for $n+1\leq j \leq m$,  and
	 $\tpq{j+1}<\tpq{j}$  for $n+1\leq j \leq m-1$.
	\begin{enumerate}[\upshape(1)]
		\item If $0 <  m$, then $(\ldots, 0, p_m, \ldots, p_2, p_1) \in
		\img(\Psi^+_{\iota^+})$.
		\item If $0 < n < m$, then
			$(\ldots, 0, q_m, \ldots, q_{n+2}, q_{n+1}, p_n, \ldots, p_2, p_1) \in
			\img(\Psi^+_{\iota^+})$.\label{enu.lem.wd1}
		\item  If $ 0 < m$, then $(\ldots, 0, q_m, \ldots, q_2, q_1) \in
		\img(\Psi^+_{\iota^-})$.
		\item  If $n < 0$, then $ (q_{0}-p_0, q_{-1}-p_{-1},\ldots, q_{n+1}-p_{n+1}, 0, \dots)\in \img(\Psi^{-}_{\iota^-})$.
		\item If $n < m < 0$, then
		$(-p_0, -p_{-1}, \ldots, -p_{m+1},  q_m-p_m, \ldots, q_{n+1}-p_{n+1}, 0, \dots) \in \img(\Psi^{-}_{\iota^-})$.\label{enu.lem.wd2}
		\item  If $n <0$, then $(-p_0, -p_{-1},\ldots, -p_{n+1}, 0, \dots)\in \img(\Psi^{-}_{\iota^-})$.
	\end{enumerate}
\end{lemma}
\begin{proof}
	We give proofs only for  parts \eqref{enu.lem.wd1} and \eqref{enu.lem.wd2};
	the proof for the other cases is easier than these cases.

	First, we show part \eqref{enu.lem.wd1}.
	By Proposition \ref{prop.impsi}, it suffices to show that
	\begin{align}
		\label{eq.wd1}
		&c_j p_j -c_{j-1} p_{j+1} \geq 0 \quad \text{ for } 2 \leq j \leq n-1,\\
		\label{eq.wd2}
		&c_j p_j -c_{j-1} q_{j+1} \geq 0 \quad \text{ for } j= n,\\
		\label{eq.wd3}
		&c_j q_j -c_{j-1} q_{j+1} \geq 0 \quad \text{ for } n+1 \leq j \leq m-1.
	\end{align}
	We can easily see by induction on $j$ that
	\begin{align}\label{eq.cp1}
		c_j p_j -c_{j-1} p_{j+1} \geq 0 \quad \text{ for } j \geq 2.
	\end{align}
	Thus we get \eqref{eq.wd1}.
	Since $\tpq{n+1} <1$, we see that $c_n p_n -c_{n-1} q_{n+1} >c_n p_n -c_{n-1} p_{n+1}$.
	Combining this inequality and \eqref{eq.cp1}, we obtain \eqref{eq.wd2}.
	For $n+1 \leq j \leq m-1$,
	we see that $c_j q_j -c_{j-1} q_{j+1} > c_j (q_{j+1} p_j / p_{j+1}) -c_{j-1} q_{j+1} =(\tpq{j+1})(c_j p_j -c_{j-1} p_{j+1})$ since $\tpq{j+1}<\tpq{j}$.
	Combining this inequality and \eqref{eq.cp1}, we obtain \eqref{eq.wd3}.
	Thus we have proved part \eqref{enu.lem.wd1}.

	Next, we show part \eqref{enu.lem.wd2}.
	Similarly, it suffices to show that
	\begin{align}
		\label{eq.wd4}
		&c_j (-p_j) -c_{j+1} (-p_{j-1}) \leq 0 \quad \text{ for } m+2 \leq j\leq -1,\\
		\label{eq.wd5}
		&c_j (-p_j) -c_{j+1} (q_{j-1}-p_{j-1}) \leq 0 \quad \text{ for } j= m+1,\\
		\label{eq.wd6}
		&c_j (q_j-p_j)  -c_{j+1} (q_{j-1}-p_{j+1}) \leq 0 \quad \text{ for } n+2 \leq j \leq m.
	\end{align}
	We can easily see 	by induction on $j$ that
	\begin{align}\label{eq.cp2}
		-c_j p_j +c_{j+1} p_{j-1} \leq 0 \quad \text{ for } j \leq -1.
	\end{align}
	Thus we get \eqref{eq.wd4}.
	We see that
	$c_{m+1} (- p_{m+1}) -c_{m+2} (q_{m}-p_m )=-c_{m+1}p_{m+1} +c_{m+2}p_m  -c_{m+2} q_{m}$.
	Combining this equality and  \eqref{eq.cp2},
	we obtain \eqref{eq.wd5}.
	For $n+2 \leq j \leq m$,
	we see that
	\begin{align}
		c_j (q_j-p_j) -c_{j+1} (q_{j-1}-p_{j-1})
		&=c_j q_j-c_{j+1} q_{j-1} -c_j p_j +c_{j-1} p_{j-1}\\
		&<c_j \bra{ \frac{q_{j-1} p_{j}}{p_{j-1}} } -c_{j+1} q_{j-1} +(-c_j p_j +c_{j+1} p_{j-1})\\
		&=\bra{1-\dpq{j-1}} (-c_jp_j +c_{j+1} p_{j-1})
	\end{align}
	since $\tpq{j}<\tpq{j-1}$.
	Combining this inequality and  \eqref{eq.cp2},
	 we obtain \eqref{eq.wd6}.
	Thus we have proved part \eqref{enu.lem.wd2}.
\end{proof}

\section{Proofs.}\label{sec.mainprf}
\subsection{Proof of Theorem \ref{thm.mainb0}.}\label{sec.mainprf1}

Let $\iota := (\iota^+, \iota^-)= ((\ldots, 2, 1, 2, 1), (2, 1, 2, 1, \ldots) )$,
and let  $\lambda=k_1\Lambda_1-k_2\Lambda_2$ be as in \S \ref{sec.maintheorems}.
We define a map $\mor$ from $\B_1(\lambda) \cup \{\0\}$ to
$\img(\Psi^+_{\iota^+}) \otimes \CT_\lambda \otimes \img(\Psi^-_{\iota^-}) \cup \{\0\}$
as follows.
First, we set $\mor(\0):=\0$.
Let
\begin{align}\label{eq.pib1re}
	\pi =
	\biggl(
	x_{m}\lambda, x_{m-1}\lambda, \ldots, x_{n}\lambda;
	0, \dfrac{q_{m}}{p_{m}}, \dfrac{q_{m-1}}{p_{m-1}}, \ldots, \dfrac{q_{n+1}}{p_{n+1}}, 1
	\biggr)\in \B_1(\lambda),
\end{align}
where $n\leq m$, and $q_m, q_{m-1}, \ldots, q_{n+1}$ are integers satisfying
$0<q_j<p_j$  for  $n+1\leq j \leq m$  and
$\tpq{j+1}<\tpq{j}$   for  $n+1\leq j \leq m-1$.
We set
\begin{align}\label{eq.mordef}
	z_k=z_k(\pi):=
	\begin{cases}
		q_k & \text{ if } 1 \leq k \text{ and } n+1 \leq k \leq m,\\
		p_k & \text{ if } 1 \leq k \text{ and } k \leq n,\\
		q_k-p_k & \text{ if } k \leq 0 \text{ and } n+1 \leq k \leq m,\\
		-p_k & \text{ if }  k \leq 0 \text{ and }  m+1 \leq k,\\
		0 & \text{ otherwise},
	\end{cases}
\end{align}
for $k \in \Z$,
and then define $\mor(\pi):=( \ldots, z_2, z_1) \otimes t_{\lambda} \otimes (z_0, z_{-1}, \ldots)\in
\Z^{+\infty}_{\iota^{+}} \otimes \CT_\lambda \otimes \Z^{-\infty}_{\iota^{-}}$.
\begin{remark}\label{rem.mor}
	More explicitly, we can describe $\mor(\pi)$ as follows:
		\begin{enumerate}[\upshape(i)]
			\item if  $n = m = 0$, that is, $\pi=\pi_\lambda$, then $\mor(\pi)= (\ldots, 0, 0) \otimes t_{\lambda} \otimes (0, 0, \ldots)$;
			\item if $0 < n = m$, then $\mor(\pi)=(\ldots, 0, p_m, \ldots, p_2, p_1) \otimes t_{\lambda} \otimes (0,0, \ldots)$;
			\item  if  $0 < n < m$, then $\mor(\pi)=  (\ldots, 0, q_m, \ldots, q_{n+2}, q_{n+1}, p_n, \ldots, p_2, p_1) \otimes t_{\lambda} \otimes (0,0, \ldots)$;
			\item  if $n = 0 < m$, then $\mor(\pi)=(\ldots, 0, q_m, \ldots, q_2, q_1) \otimes t_{\lambda} \otimes (0,0, \ldots)$;
			\item if $n < 0 < m$, then  $ \mor(\pi)=(\ldots, 0, q_m, \ldots, q_2, q_1) \otimes t_{\lambda} \otimes (q_0-p_0, q_{-1}-p_{-1}, \ldots, q_{n+1}-p_{n+1}, 0, \dots)$;
			\item  if $n < m = 0$, then $\mor(\pi)=(\ldots, 0, 0) \otimes t_{\lambda} \otimes (q_{0}-p_0, q_{-1}-p_{-1},\ldots, q_{n+1}-p_{n+1}, 0, \dots)$;
			\item if $n < m < 0$, then $\mor(\pi)=(\ldots, 0, 0) \otimes t_{\lambda} \otimes (-p_0, -p_{-1}, \ldots, -p_{m+1},  q_m-p_m, q_{m-1}-p_{m-1}, \ldots, q_{n+1}-p_{n+1}, 0, \dots)$;
			\item  if $n = m <0$, then $\mor(\pi)= (\ldots, 0, 0) \otimes t_{\lambda} \otimes (-p_0, -p_{-1},\ldots, -p_{n+1}, 0, \dots)$.
		\end{enumerate}
	Therefore, by Lemma \ref{lem.wd}, we deduce that  $\mor(\pi) \in \img(\Psi^+_{\iota^+}) \otimes \CT_\lambda \otimes \img(\Psi^-_{\iota^-})$
	for $\pi \in \B_1(\lambda)$.
\end{remark}

\begin{theorem}\label{thm.mor}
	The map $\mor:\B_1(\lambda) \rightarrow \img(\Psi^+_{\iota^+}) \otimes \CT_\lambda \otimes \img(\Psi^-_{\iota^-})$
	is an embedding of crystals.
\end{theorem}
%
Assuming that Theorem \ref{thm.mor} is true,
we give a proof of Theorem  \ref{thm.mainb0}.
\begin{proof}[\it{Proof of Theorem \ref{thm.mainb0}}]
	Let $Z(\lambda):=\{ b \in \img(\Psi^+_{\iota^+}) \otimes \CT_\lambda \otimes \img(\Psi^-_{\iota^-}) \mid b^\ast  \text{ is extremal} \}.$
	We know from Theorem \ref{thm.cbext} and Corollary \ref{cor.polyhofex}
	that
	there exists an isomorphism $\Sigma: Z(\lambda) \rightarrow \CB(\lambda) $,
	which sends  $z_\lambda:=(\ldots, 0, 0) \otimes t_{\lambda} \otimes (0, 0, \ldots)$ to $u_\lambda$.
	Recall from \eqref{eq.b01} that $\B_0(\lambda)\subset \B_1(\lambda) $.
	Because $\mor(\pi_\lambda)=z_\lambda \in Z(\lambda)$, we see that
	$\mor(\B_0(\lambda))\subset Z(\lambda)$.
	Therefore
	it follows from Theorem \ref{thm.mor} that
	$\Sigma  \circ \left. \mor \right|_{\B_0(\lambda)} $ is an isomorphism of crystals from $\B_0(\lambda)$ onto $\CB_0(\lambda)$.
	Thus we have proved Theorem \ref{thm.mainb0}.
\end{proof}

The rest of this subsection is devoted to a proof of Theorem \ref{thm.mor}.
%
\begin{lemma}\label{lem.xkl}
	For $k \leq l$, it holds that
	\begin{align}
		x_k\lambda-x_l\lambda=\sum_{j=k+1}^l p_j \al_{i_j}.
	\end{align}
\end{lemma}
\begin{proof}
	We proceed by induction on $l$; recall that  $l \geq k$.
	If $l=k$, then the assertion is obvious.
	Assume that $l > k$.
	By the induction hypothesis,
	we have $x_k\lambda-x_{l-1}\lambda=\sum_{j=k+1}^{l-1} p_j \al_{i_j}$.
	We see by \eqref{eq.xm} that
	$x_l\lambda=x_{l-1}\lambda - p_l\al_{i_l}$.
	Therefore, we obtain
	\begin{align}
		x_k\lambda-x_l\lambda =x_k\lambda-x_{l-1}\lambda + p_l\al_{i_l}
		=\sum_{j=k+1}^{l-1} p_{j} \al_{i_j} + p_l\al_{i_l}
		= \sum_{j=k+1}^l p_j \al_{i_j},
	\end{align}
	as desired.
\end{proof}

\begin{proposition}\label{prop.wt}
	Let $\pi \in \B_1(\lambda)$ be as \eqref{eq.pib1re}.
	Then,
	\begin{align}
		\wt{(\pi)}
		=\wt(\mor(\pi))
		=x_n\lambda-\sum_{j=n+1}^m q_j \al_{i_j}.
	\end{align}
\end{proposition}

\begin{proof}
	First, we show by induction on $m$ that $\wt{(\pi)} =x_n\lambda-\sum_{j=n+1}^m q_j \al_{i_j}$; recall that $m \geq n$.
	If $m=n$, then $\wt{(\pi)} =x_n\lambda$ since $\pi=(x_n \lambda; 0, 1)$.
	Hence the  assertion is obvious.
	Assume that $m>n$.
	We see that
	\begin{align}
		\pi':= \bra{x_{m-1}\lambda, x_{m-2}\lambda, \ldots, x_{n}\lambda; 0, \dpq{m-1}, \dpq{m-2}, \ldots, \dpq{n+1}, 1 }
	\end{align}
	is also  an element of  $\B_1(\lambda)$.
	By the induction hypothesis, we obtain
	$\wt{(\pi')} =x_n\lambda-\sum_{j=n+1}^{m-1} q_j \al_{i_j}$.
	We see by the definition of $\wt$ that
	\begin{align}
		\wt(\pi)& =\wt(\pi')-\dpq{m-1}x_{m-1}\lambda+\dpq{m}x_{m}\lambda+\bra{\dpq{m-1}-\dpq{m}}x_{m-1}\lambda\\
		&=\wt(\pi')+\dpq{m}(x_m\lambda-x_{m-1}\lambda).
	\end{align}
	We see from Lemma \ref{lem.xkl} that $x_m\lambda-x_{m-1}\lambda=-p_m\al_{i_m}$.
	Therefore we deduce that
	$\wt{(\pi)} =x_n\lambda-\sum_{j=n+1}^{m-1} q_j \al_{i_j} +(\tpq{m})(-p_m\al_{i_m})
	=x_n\lambda-\sum_{j=n+1}^{m} q_j \al_{i_j} $.

	Next, we show that
	$\wt(\mor(\pi)) =x_n\lambda-\sum_{j=n+1}^m q_j \al_{i_j}$.
	By the definition of $\wt$,
	if $0\leq n \leq m $,
	then we have
	\begin{align}
		\wt(\mor(\pi))
		& =\lambda-\sum_{j=n+1}^m q_j\al_{i_j}-\sum_{j=1}^n p_j \al_{i_j};
	\end{align}
	if $n < 0 < m $, then we have
	\begin{align}
	\wt(\mor(\pi))
		=  \lambda-\sum_{j=1}^m q_j \al_{i_j} - \sum_{j=n+1}^0 (q_j-p_j) \al_{i_j}
		=  \lambda- \sum_{j=n+1}^m q_j\al_{i_j} + \sum_{j=n+1}^0 p_j \al_{i_j};
	\end{align}
		if $n \leq m \leq 0 $, then we have
	\begin{align}
			\wt(\mor(\pi))
			=  \lambda - \sum_{j=m+1}^0 (-p_j) \al_{i_j}- \sum_{j=n+1}^m (q_j-p_j)\al_{i_j}
			= \lambda+\sum_{j=n+1}^0 p_j\al_{i_j}-\sum_{j=n+1}^m q_j\al_{i_j}.
	\end{align}
	It follows from Lemma \ref{lem.xkl} that
	\begin{align}
		x_n\lambda=
		\begin{cases}
			\lambda-\sum_{j=1}^n p_j \al_{i_j}& \text{ if } n \geq 0,\\
			\lambda+\sum_{j=n+1}^0 p_j \al_{i_j} & \text{ if } n \leq 0.
		\end{cases}
	\end{align}
	Therefore we obtain $\wt(\mor(\pi)) =x_n\lambda-\sum_{j=n+1}^m q_j \al_{i_j}$
	for $n, m \in \Z$ such that $n\leq m$.
	Thus we have proved the proposition.
\end{proof}

\begin{lemma}\label{lem.sig}
	Let $\pi \in \B_1(\lambda)$ be as \eqref{eq.pib1re}.
	Then, for $k\in \Z$,
	\begin{align}
		\sigma_k(\mor(\pi)) =
		\begin{cases}
			0 & \text{\rm if } m+1 \leq k,\\
			q_{k}+\sum_{j=k+1}^m \pair{\al_{i_j}}{\alc_{i_k}}q_j & \text{\rm if } n+1 \leq k \leq m,\\
			-\pair{\wt(\mor(\pi))}{\alc_{i_k}} & \text{\rm if } k \leq n.
		\end{cases}
	\end{align}
\end{lemma}

\begin{proof}
	First, we assume that $0\leq n \leq m $.
	We write
	$\mor(\pi)= b_1 \otimes t_{\lambda} \otimes (0,0, \ldots)$
	with
	$b_1= (\ldots, 0, q_m, \ldots, q_{n+2}, q_{n+1}, p_n, \ldots, p_2, p_1)$,
	where we understand
	$b_1= (\ldots, 0, p_m, \ldots, p_{2}, p_{1})$
	\resp{$(\ldots, 0, q_m, \ldots, q_{2}, q_{1})$}
	if $0 < n = m$ \resp{$n = 0 < m$}.
	If $n+1 \leq k$, then we have
		$\sigma_k(\mor(\pi))=\sigma_k^+(b_1)$.
	Hence the assertion is obvious by the definition of $\sigma_k^+$.
	Assume that  $1 \leq k \leq n$.
	By Proposition \ref{prop.wt},
	it suffices to show that
	$\sigma_k(\mor(\pi))=-\pair{x_n\lambda}{\alc_{i_k}}
	+\sum_{j=n+1}^m \pair{\al_{i_j}}{\alc_{i_k}}q_j$.
	We see by the definition of $\sigma_k$ that
	\begin{align}\label{eq.122001}
		\sigma_k(\mor(\pi))=
		\sigma_k^+(b_1) = p_{k}+ \sum_{j=k+1}^n \pair{\al_{i_j}}{\alc_{i_k}}p_j
		+\sum_{j=n+1}^m \pair{\al_{i_j}}{\alc_{i_k}}q_j.
	\end{align}
	It follows from Lemma \ref{lem.xkl} that
	$x_k\lambda-x_n\lambda=\sum_{j=k+1}^n p_{j}\al_{i_j}$.
	Therefore,
	\begin{align}\label{eq.122002}
		-\pair{x_n\lambda}{\alc_{i_k}}
		=-\pair{x_k\lambda}{\alc_{i_k}} + \sum_{j=k+1}^n \pair{\al_{i_j}}{\alc_{i_k}}p_j
		=p_k+\sum_{j=k+1}^n \pair{\al_{i_j}}{\alc_{i_k}}p_j
	\end{align}
	since $\pair{x_k\lambda}{\alc_{i_k}}=-p_k$ by \eqref{eq.xm}.
	Combining \eqref{eq.122001} and \eqref{eq.122002},
	we obtain the desired equality.
	If $k \leq 0$, then we have
	$\sigma_k(\mor(\pi))=\sigma_k^-((0,0,\ldots))-\pair{\wt(\mor(\pi))}{\alc_{i_k}} =-\pair{\wt(\mor(\pi))}{\alc_{i_k}}$.

	Next, we assume that $n < 0 < m$.
	We write
	$\mor(\pi)= b_1 \otimes t_{\lambda} \otimes b_2$
	with
	$b_1= (\ldots, 0, q_m, \ldots, q_{2}, q_{1})$ and
	$b_2= (q_0-p_0, q_{-1}-p_{-1}, \ldots,  q_{n+1}-p_{n+1}, 0, \ldots)$.
	If $1 \leq k$, then $\sigma_k(\mor(\pi))=\sigma_k^+(b_1)$.
	Hence the assertion is obvious
	by the definition of $\sigma_k^+$.
	Assume that  $n+1 \leq k \leq 0$.
	We see by the definition of $\sigma_k^-$ that
	\begin{align}
		\sigma_k^-(b_2)
		&=-(q_{k}-p_k)- \sum_{j=n+1}^{k-1} \pair{\al_{i_j}}{\alc_{i_k}}(q_j-p_j)\\
		&=-q_k+ p_k
		-\sum_{j=n+1}^{k-1} \pair{\al_{i_j}}{\alc_{i_k}}q_j
		+\sum_{j=n+1}^{k-1} \pair{\al_{i_j}}{\alc_{i_k}}p_j.\label{eq.lemsigtest4}
	\end{align}
	By Proposition \ref{prop.wt}, we have
	$\pair{\wt(\mor(\pi))}{\alc_{i_k}}
			=\pair{x_n\lambda}{\alc_{i_k}}-\sum_{j=n+1}^m \pair{\al_{i_j}}{\alc_{i_k}}q_j$.
	Hence,
	\begin{align}
		\sigma_k(\mor(\pi))
		&= \sigma_k^-(b_2)-\pair{\wt(\mor(\pi))}{\alc_{i_k}}\\
		&=-q_k
		+\sum_{j=k}^{m} \pair{\al_{i_j}}{\alc_{i_k}}q_j
		+ p_k
		+\sum_{j=n+1}^{k-1} \pair{\al_{i_j}}{\alc_{i_k}}p_j
		-\pair{x_n\lambda}{\alc_{i_k}}.\label{eq.122003}
	\end{align}
	Because $ \pair{\al_{i_k}}{\alc_{i_k}}=2$, we obtain
	\begin{align}
		-q_k
		+\sum_{j=k}^{m} \pair{\al_{i_j}}{\alc_{i_k}}q_j
		&=-q_k
		+\pair{\al_{i_k}}{\alc_{i_k}}q_k
		+\sum_{j=k+1}^{m} \pair{\al_{i_j}}{\alc_{i_k}}q_j\\
		&=q_k
		+\sum_{j=k+1}^{m} \pair{\al_{i_j}}{\alc_{i_k}}q_j.\label{eq.122004}
	\end{align}
	It follows from Lemma \ref{lem.xkl} that
	$-x_n\lambda +x_k\lambda +\sum_{j=n+1}^k p_j\al_{i_j}=0$,
	and hence,
	\begin{align}
		0
		&=-\pair{x_n\lambda}{\alc_{i_k}}
		+\pair{x_k\lambda}{\alc_{i_k}}
		+ \sum_{j=n+1}^k \pair{\al_{i_j}}{\alc_{i_k}}p_j\\
		&=-\pair{x_n\lambda}{\alc_{i_k}}
		+p_k
		+ \sum_{j=n+1}^{k-1} \pair{\al_{i_j}}{\alc_{i_k}}p_j\label{eq.122005}
	\end{align}
	since $\pair{x_k\lambda}{\alc_{i_k}}=-p_k$
	and $\pair{\al_{i_k}}{\alc_{i_k}}=2$.
	By \eqref{eq.122003}--\noeqref{eq.122004}\eqref{eq.122005}, we obtain
	$\sigma_k(\mor(\pi))
	 =q_{k}+\sum_{j=k+1}^{m} \pair{\al_{i_j}}{\alc_{i_k}}q_j$,
	as desired.
	If $k \leq n$, then  $\sigma_k^-(b_2)=0$, which implies that
	$\sigma_k(\mor(\pi))=\sigma_k^-(b_2)-\pair{\wt(\mor(\pi))}{\alc_{i_k}} =-\pair{\wt(\mor(\pi))}{\alc_{i_k}}$.

	Finally, we assume that $n \leq m \leq 0$.
	We write
	$\mor(\pi)= (\ldots, 0, 0) \otimes t_{\lambda} \otimes b_2$
	with
		$b_2= (-p_0, -p_{-1}, \ldots,  -p_{m+1}, q_{m}-p_{m}, \ldots, q_{n+1}-p_{n+1}, 0, \ldots)$,
	where we understand
	$b_2= (q_0-p_0, q_{-1}-p_{-1}, \ldots,  q_{n+1}-p_{n+1}, 0, \ldots)$
	\resp{$b_2= (-p_0, -p_{-1}, \ldots,  -p_{n+1}, 0, \ldots)$}
	if $n < m = 0$ \resp{$n = m < 0$}.
	If $1 \leq k$, then it is obvious that  $\sigma_k(\mor(\pi))=\sigma_k^+(b_1)=0$.
	Assume that $m+1 \leq k \leq 0$.
	We see that
	\begin{align}
		\sigma_k^-(b_2)
		&=-(-p_k) - \sum_{j=m+1}^{k-1} \pair{\al_{i_j}}{\alc_{i_k}}(-p_j)
		- \sum_{j=n+1}^{m} \pair{\al_{i_j}}{\alc_{i_k}}(q_j-p_j)\\
		&= p_k +\sum_{j=n+1}^{k-1} \pair{\al_{i_j}}{\alc_{i_k}}p_j
		-\sum_{j=n+1}^{m} \pair{\al_{i_j}}{\alc_{i_k}}q_j.\label{eq.lemsigtest1}
	\end{align}
	By Proposition \ref{prop.wt}, we have
	$\pair{\wt(\mor(\pi))}{\alc_{i_k}}
			=\pair{x_n\lambda}{\alc_{i_k}}-\sum_{j=n+1}^m \pair{\al_{i_j}}{\alc_{i_k}}q_j$.
	Hence,
	\begin{align}
		\sigma_k(\mor(\pi))
		= \sigma_k^-(b_2)-\pair{\wt(\mor(\pi))}{\alc_{i_k}}
		= p_k
		+\sum_{j=n+1}^{k-1} \pair{\al_{i_j}}{\alc_{i_k}}p_j
		-\pair{x_n\lambda}{\alc_{i_k}}.\label{eq.122007}
	\end{align}
	It follows from Lemma \ref{lem.xkl} that
	$-x_n\lambda +x_k\lambda +\sum_{j=n+1}^k p_j\al_{i_j}=0$,
	and hence,
	\begin{align}
		0
		&=-\pair{x_n\lambda}{\alc_{i_k}}
		+\pair{x_k\lambda}{\alc_{i_k}}
		+ \sum_{j=n+1}^k \pair{\al_{i_j}}{\alc_{i_k}}p_j\\
		&=-\pair{x_n\lambda}{\alc_{i_k}}
		+p_k
		+ \sum_{j=n+1}^{k-1} \pair{\al_{i_j}}{\alc_{i_k}}p_j\label{eq.122008}
	\end{align}
	since $\pair{x_k\lambda}{\alc_{i_k}}=-p_k$
	 and $\pair{\al_{i_k}}{\alc_{i_k}}=2$.
	By \eqref{eq.122007} and  \eqref{eq.122008},
	we obtain
	$\sigma_k(\mor(\pi))=\sigma_k^-(b_2)-\pair{\wt(\mor(\pi))}{\alc_{i_k}}=0$, as desired.
	If $ k \leq m $,
	then we can show the equality
	by the same argument as in the case that  $n < 0 < m $.
	Thus we have proved the lemma.
\end{proof}

Now, we set
\begin{align}
	i(k):=i_k=
	\begin{cases}
		2 & \text{ if } k \text{ is even},\\
		1 & \text{ if } k \text{ is odd},
	\end{cases}
	\quad
	i'(k):=
	\begin{cases}
		1 & \text{ if } k \text{ is even},\\
		2 & \text{ if } k \text{ is odd},
	\end{cases}
\end{align}
for $k \in\Z$;
note that
\begin{align}\label{eq.idir}
	\pair{x_k\lambda}{\alc_{i(k)}}=-p_k <0, \quad
	\pair{x_k\lambda}{\alc_{i'(k)}}=p_{k+1} >0
\end{align}
by \eqref{eq.xm}.
Let us write $\pi \in \B_1(\lambda)$ as \eqref{eq.pib1re}.
We see by \eqref{eq.idir} that
\begin{align}\label{eq.minval1}
	H^\pi_{i(m)}\bra{\dpq{m}}<0=H^\pi_{i(m)}(0).
\end{align}
Moreover,
if $m+n$ is odd,  then
we see that $\pair{x_n\lambda}{\alc_{i(m)}}>0$,
and hence
\begin{align}
	H^\pi_{i(m)}\bra{\dpq{n+1}}<\pair{\wt(\pi)}{\alc_{i(m)}}=H^\pi_{i(m)}(1)\label{eq.minval2}.
\end{align}
If $m+n$ is even,
then we see that
$\pair{x_n\lambda}{\alc_{i'(m)}}>0$,
and hence
\begin{align}
 H^\pi_{i'(m)}\bra{\dpq{n+1}}<\pair{\wt(\pi)}{\alc_{i'(m)}}=H^\pi_{i'(m)}(1).\label{eq.minval3}
\end{align}


\begin{lemma}\label{lem.minval}
	Let $\pi \in \B_1(\lambda)$ be as \eqref{eq.pib1re}.
	If  $n+1 \leq k \leq m$, then
	\begin{align}
		-q_k+\sum_{j=k+1}^m\pair{\al_{i_j}}{\alc_{i_k}}q_j=
		\begin{cases}
			H^\pi_{i(m)}(\tpq{k}) & \text{\rm if } k-m \in 2 \Z,\\
			H^\pi_{i'(m)}(\tpq{k})  & \text{\rm if } k-m+1 \in 2 \Z.
		\end{cases}
	\end{align}
\end{lemma}
\begin{proof}
	We set $q_{m+1}:=0$ and $q_{n}:=p_n$ by convention.
	Assume that $k-m \in 2 \Z$. Then we obtain
	\begin{align}
		H^\pi_{i(m)}\bra{\dpq{k}}
		&=\sum_{j=k}^m
		\bra{\dpq{j}-\dpq{j+1}}\pair{x_{j}\lambda}{\alc_{i(m)}}\\
		&=\sum_{j=k, k+2, \ldots, m-2}
		\bra{
			\bra{\dpq{j}-\dpq{j+1}}(-p_j)
			+\bra{\dpq{j+1}-\dpq{j+2}}p_{j+2}
		}
		+\dpq{m}(-p_m)\\
		&=\sum_{j=k, k+2, \ldots, m-2} (-q_{j}+a_{i(m)}q_{j+1}-q_{j+2})-q_m
			\qquad \text{ by } \eqref{eq.pm} \text{ and } \eqref{eq.pm2}\\
		&=-q_k +  \sum_{j=k, k+2, \ldots, m-2} (a_{i(m)}q_{j+1}-2q_{j+2})\\
		&=-q_k+\sum_{j=k+1}^m\pair{\al_{i_j}}{\alc_{i(m)}}q_j,
	\end{align}
	as desired.

	Assume that $k-m+ 1 \in 2 \Z$. Then we obtain
	\begin{align}
		H^\pi_{i'(m)}\bra{\dpq{k}}
		&=\sum_{j=k}^m
		\bra{\dpq{j}-\dpq{j+1}}\pair{x_{j}\lambda}{\alc_{i'(m)}}\\
		&=\sum_{j=k, k+2, \ldots, m-1}
		\bra{
			\bra{\dpq{j}-\dpq{j+1}}(-p_j)+\bra{\dpq{j+1}-\dpq{j+2}}p_{j+2}
		}\\
		&=\sum_{j=k, k+2, \ldots, m-1} (-q_{j}+a_{i'(m)}q_{j+1}-q_{j+2})
			\qquad \text{ by } \eqref{eq.pm} \text{ and } \eqref{eq.pm2}\\
		&=-q_k+\sum_{j=k, k+2, \ldots, m-3} (a_{i'(m)}q_{j+1}-2q_{j+2})+a_{i'(m)}q_m-q_{m+1}\\
		&=-q_k+\sum_{j=k+1}^m\pair{\al_{i_j}}{\alc_{i'(m)}}q_j,
	\end{align}
		as desired.
\end{proof}
By Lemmas \ref{lem.sig} and \ref{lem.minval},
we obtain the following proposition.
\begin{proposition}\label{prop.sigminval}
	Let $\pi \in \B_1(\lambda)$ be as \eqref{eq.pib1re}. Then,
	\begin{align}\label{eq.sigma}
		-\sigma_k(\mor(\pi)) =
		\begin{cases}
			H^\pi_{i(m)}(0) & \text{\rm if }  k-m \in 2\Z \text{\rm{ and }}  m+1 \leq k,\\
			H^\pi_{i(m)}(\tpq{k}) & \text{\rm if } k-m \in 2 \Z \text{\rm{ and }} n+1 \leq k \leq m,\\
			H^\pi_{i(m)}(1) & \text{\rm if } k-m \in 2 \Z \text{\rm{ and }} k \leq n,\\
			H^\pi_{i'(m)}(0) & \text{\rm if } k-m+1 \in 2 \Z \text{\rm{ and }} k \leq m+1,\\
			H^\pi_{i'(m)}(\tpq{k}) & \text{\rm if } k-m+1 \in 2 \Z \text{\rm{ and }} n+1 \leq k \leq m,\\
			H^\pi_{i'(m)}(1) & \text{\rm if } k-m+1 \in 2 \Z \text{\rm{ and }} k \leq n.
		\end{cases}
	\end{align}
\end{proposition}

\begin{proof}[\it{Proof of Theorem \ref{thm.mor}}]
	By Remark \ref{rem.mor}, it is easy to check that the map $\mor$ is injective.
	We show that $\mor$ is a morphism of crystals.
	Let $\pi \in \B_1(\lambda)$.
	We have $\wt(\pi)=\wt(\mor(\pi))$ by Proposition \ref{prop.wt}.
	We show that
	$\ep_i(\pi)=\ep_i(\mor(\pi))$ and 	$\mor(\e_i\pi)=\e_i\mor(\pi)$ for $i \in I$.
	Let us write $\pi$ as \eqref{eq.pib1re}.

	{\bf{Case 1}}. Assume that $i=i(m)$ and $m+n$ is odd.
	Note that the function $H_{i}^\pi(t)$ attains a minimal value
	at $t=\tpq{k}$, $k=m ,m-2, \ldots,n+1$ (see Remark \ref{rem.alt} and \eqref{eq.idir}).
	By \eqref{eq.LSepph} we have
	\begin{align}
		\ep_{i}(\pi)
		&=-\min
			\left\{H_{i}^\pi(t) \, \middle| \,
				t \in \left\{
					\dpq{m}, \dpq{m-2}, \ldots,  \dpq{n+1}
				\right\}
			\right\}\\
			&=\max
				\left\{-H_{i}^\pi(t) \, \middle| \,
					t \in \left\{
						\dpq{m}, \dpq{m-2}, \ldots,  \dpq{n+1}
					\right\}
				\right\}.\label{eq.mor1}
	\end{align}
	By the definition of $\ep_{i}(\mor(\pi))$, we have
	\begin{align}\label{eq.mor3}
		\ep_{i}(\mor(\pi))
		 =\max_{k:  i_k=i}\sigma_k(\mor(\pi))
		=\max_{k-m \in 2\Z}\sigma_k(\mor(\pi)).
	\end{align}
	We see from  Proposition \ref{prop.sigminval} that
	\begin{align}
		\max_{k-m \in 2\Z}\sigma_k(\mor(\pi))
		&=\max
			\left\{-H_{i}^\pi(t) \, \middle| \,
				t \in \left\{
					0, \dpq{m}, \dpq{m-2}, \ldots,  \dpq{n+1}, 1
				\right\}
			\right\}\\
			&=\max
				\left\{-H_{i}^\pi(t) \, \middle| \,
					t \in \left\{
						\dpq{m}, \dpq{m-2}, \ldots,  \dpq{n+1}
					\right\}
				\right\},\label{eq.mor2}
	\end{align}
	where the second equality follows from \eqref{eq.minval1} and \eqref{eq.minval2}.
	By \eqref{eq.mor1}--\eqref{eq.mor2},
	\noeqref{eq.mor3}
	we obtain $\ep_{i}(\pi)=\ep_{i}(\mor(\pi))$, as desired.
	Next, we show that  $\mor(\e_i\pi)=\e_i\mor(\pi)$.
	Since both $\B_1(\lambda)$ and $\img(\Psi^+_{\iota^+}) \otimes \CT_\lambda \otimes \img(\Psi^-_{\iota^-})$ are normal crystals in the sense of \cite{K},
	the equality $\ep_i(\pi)=\ep_i(\mor(\pi))$ and the injectivity  of $\mor$ imply that
	\begin{align}
		\mor(\e_i\pi)=\0 \iff \e_i\pi=\0
		\iff \ep_i(\pi)=0
		\iff \ep_i(\mor(\pi))=0
		\iff \e_i\mor(\pi)=\0.
	\end{align}
	Assume that $\e_i\pi\neq \0$, or equivalently, $\e_i\mor(\pi) \neq \0$.
	By the definition of $\mor$,
	we have
	 $\mor(\pi)=(\ldots, y_2, y_1) \otimes t_\lambda \otimes (y_0,y_{-1}, \ldots)$,
	where  $y_k=z_k(\pi)$ (see \eqref{eq.mordef}).
	Let $M_{(i)}$ be as \eqref{eq.mi},
	and set $k' :=\max M_{(i)}$.
	Namely, $k'$ is
 	the largest integer $k$ such that
	$\sigma_{(i)}(\mor(\pi))=\sigma_k(\mor(\pi))$ and $k-m \in 2\Z$.
	Then we see by the definition of $\e_i$
 that $\e_i\mor(\pi) =(\ldots, y'_2, y'_1) \otimes t_\lambda \otimes (y'_0,y'_{-1}, \ldots)$, where $y'_k:=y_k - \delta_{k, k'}$.
	Let $t_1$ and $t_0$ be as \eqref{eq.et1} and \eqref{eq.et0}, respectively.
	By \eqref{eq.mor1}--\eqref{eq.mor2}, we obtain
	$t_1=\tpq{k'}$.
	By \eqref{eq.estrict} and Remark \ref{rem.alt},
	we have $t_0= t_1-1/(-\pair{x_{k'}\lambda}{\alc_{i}})=(q_{k'}-1)/p_{k'}$.
	Assume that  $k'<m$.
	By \eqref{eq.estrict} and Remark \ref{rem.alt},
	we have $ \tpq{k'+1} \leq t_0$.
	Suppose, for a contradiction, that  $\tpq{k'+1}= t_0$, that is,
	\begin{align}\label{eq.case1contra}
		H_i^\pi\bra{\dpq{k'+1}}=H_i^\pi\bra{\dpq{k'}}+1.
	\end{align}
	Then it follows from   Remark \ref{rem.alt}
	that $H_i^\pi(t)$ attains a minimal value at $t=\tpq{k'+2}$,
	and hence $H_i^\pi(\tpq{k'+2}) \in \Z$ by \eqref{eq.int}.
	By \eqref{eq.case1contra}, we obtain $H_i^\pi(\tpq{k'+2}) \leq H_i^\pi(\tpq{k'})$,
	which contradicts the definition of $t_1$.
	Therefore we obtain $ \tpq{k'+1} <t_0$ and
	\begin{align}
		\e_{i}\pi=\bra{x_m\lambda, \ldots, x_{k'}\lambda \ldots, x_n\lambda; 0, \dpq{m} \ldots, \dpq{k'+1}, \frac{q_{k'}-1}{q_{k'}}, \dpq{k'-1}, \ldots, \dpq{n+1}, 1}.
	\end{align}
	If $k'=m$, then
	\begin{align}
		\e_{i}\pi=
		\begin{dcases}
			\bra{ x_m\lambda, \ldots, x_n\lambda; 0, \frac{q_{m}-1}{q_{m}},  \dpq{m-1}, \ldots, \dpq{n+1}, 1} & \text{ if } q_m >1,   \\
		\bra{ x_{m-1}\lambda, \ldots, x_n\lambda; 0,  \dpq{m-1}, \ldots, \dpq{n+1}, 1} & \text{ if } q_m=1.
		\end{dcases}
	\end{align}
	Hence we see that
	\begin{align}
		z_k(\e_i\pi)
		&=
		\begin{cases}
			q_{k'}-1 & \text{ if } 1 \leq k=k', \\
			(q_{k'}-1)-p_{k'} & \text{ if } k=k'  \leq 0,\\
			q_k & \text{ if } k\neq k', \, 1 \leq k, \text{ and } n+1 \leq k \leq m,\\
			p_k & \text{ if } k\neq k', \, 1 \leq k, \text{ and } k \leq n,\\
			q_k-p_k & \text{ if } k\neq k', \, k \leq 0, \text{ and } n+1 \leq k \leq m,\\
			-p_k & \text{ if } k\neq k', \,  k \leq 0, \text{ and }  m+1 \leq k,\\
			0 & \text{ otherwise},
		\end{cases}\\
		&= z_k(\pi)-\delta_{k, k'},
	\end{align}
	which implies that  $\mor(\e_{i}\pi)=\e_{i}\mor(\pi)$.


	{\bf{Case 2}}. Assume that $i=i'(m)$ and $m+n$ is even.
	Note that the function $H_{i}^\pi(t)$ attains a minimal value
	at $t=0$ and $t=\tpq{k}$, $k=m-1, m-3, \ldots, n+1$.
	As in Case 1, we deduce by
	Proposition \ref{prop.sigminval} and \eqref{eq.minval3} that
	\begin{align}
		\ep_{i}(\pi)
		&=-\min
			\left\{H_{i}^\pi(t) \, \middle| \,
				t \in \left\{
					0, \dpq{m-1}, \dpq{m-3}, \ldots,  \dpq{n+1}
				\right\}
			\right\}\\
		&=\max
			\left\{-H_{i}^\pi(t) \, \middle| \,
				t \in \left\{
					0, \dpq{m-1}, \dpq{m-3}, \ldots,  \dpq{n+1}
				\right\}
			\right\}\\
		&=\max_{k-m+1 \in 2\Z}\sigma_k(\mor(\pi))= \ep_{i}(\mor(\pi)).
	\end{align}
	We can show that $\mor(\e_{i}\pi)=\e_{i}\mor(\pi)$
	in exactly the same way as Case 1.

	{\bf{Case 3}}. Assume that $i=i(m)$ and $m+n$ is even.
	Note that the function $H_{i}^\pi(t)$ attains a minimal value
	at $t=\tpq{k}$, $k=m, m-2, \ldots, n+2$ and $t=1$.
	As in Case 1,
	we deduce by Proposition \ref{prop.sigminval} and \eqref{eq.minval1} that
	\begin{align}
		\ep_{i}(\pi)
		&=-\min
			\left\{H_{i}^\pi(t) \, \middle| \,
				t \in \left\{
					\dpq{m}, \dpq{m-2}, \ldots,  \dpq{n+2}, 1
				\right\}
			\right\}\\
			&=\max
				\left\{-H_{i}^\pi(t) \, \middle| \,
					t \in \left\{
						\dpq{m}, \dpq{m-2}, \ldots,  \dpq{n+2}, 1
					\right\}
				\right\}\\
			&=\max_{k-m \in 2\Z}\sigma_k(\mor(\pi))=\ep_{i}(\mor(\pi)).\label{eq.mor4}
	\end{align}
	We show that  $\mor(\e_i\pi)=\e_i\mor(\pi)$.
	If $m=n$, then $\pi= (x_n\lambda; 0, 1 )$.
	We see by definition of $\mor $ that
	\begin{align}
		\mor(\pi)=
		\begin{cases}
			(\ldots, 0, p_n, \ldots, p_2, p_1) \otimes t_{\lambda} \otimes (0,0, \ldots)
			& \text{ if } n >0,\\
			(\ldots, 0, 0) \otimes t_{\lambda} \otimes (0, 0, \ldots)
			& \text{ if } n =0,\\
			(\ldots, 0, 0) \otimes t_{\lambda} \otimes (-p_0, -p_{-1},\ldots, -p_{n+1}, 0, \dots)
			& \text{ if } n < 0.
		\end{cases}
	\end{align}
	Also, we see that
	\begin{align}
		\e_i\pi=
		\begin{cases}
			\bra{x_n\lambda, x_{n-1}\lambda; 0, {(p_n-1)}/{p_n}, 1 } & \text{ if } p_n >1,\\
			\bra{x_{n-1}\lambda; 0, 1 }& \text{ if } p_n =1.
		\end{cases}
	\end{align}
	Thus it is easy to verify that $\mor(\e_i\pi)=\e_i\mor(\pi)$ in this case.
	Assume that $m>n$;
	by the  assumption that $m+n$ is even, we have $m \geq n+2$.
	Let $M_{(i)}$ be as \eqref{eq.mi},
	and set $k':=\max M_{(i)}$.
	If $k' \in \{ m, m-2, \ldots, n+2 \}$,
	then we can show in exactly the same way as Case 1
	that $\mor(\e_{i}(\pi))=\e_{i}\mor(\pi)$.
	Otherwise,
	we see by Proposition \ref{prop.sigminval}  and \eqref{eq.mor4} that
	$k'=n$. Let $\mor(\pi)=(\ldots, y_2, y_1) \otimes t_\lambda \otimes (y_0,y_{-1}, \ldots)$, where  $y_k=z_k(\pi)$.
	Then we see by the definition of $\e_i$ that
	$\e_i\mor(\pi)=(\ldots, y'_2, y'_1) \otimes t_\lambda \otimes (y'_0,y'_{-1}, \ldots)$,
 where  $y'_k=y_k-\delta_{k, n}$.
 	Let $t_1$ and $t_0$ be as \eqref{eq.et1} and \eqref{eq.et0}, respectively.
	We see that $t_1=1$ and $t_0=1-1/(-\pair{x_n\lambda}{\alc_{i}})=1-1/p_{n-1}$.
	By \eqref{eq.estrict} and Remark \ref{rem.alt},
	we have $ \tpq{n+1} \leq t_0$.
	Suppose, for a contradiction, that  $\tpq{n+1}= t_0$; note that
		$H_i^\pi(\tpq{n+1})=H_i^\pi(1)+1$.
	It follows from  Remark \ref{rem.alt}
	that $H_i^\pi(t)$ attains a minimal value at $t=\tpq{n+2}$,
	and hence $H_i^\pi(\tpq{n+2}) \in \Z$ by \eqref{eq.int}.
	Therefore, we obtain $H_i^\pi(\tpq{n+2}) \leq H_i^\pi(1)$,
	which contradicts the definition of $t_1$.
	Therefore we obtain $ \tpq{n+1} <t_0$, which implies that
	\begin{align}
		\e_{i}\pi=\bra{x_m\lambda, \ldots, x_n\lambda,  x_{n-1}\lambda ; 0, \dpq{m} \ldots, \ldots, \dpq{n+1},\frac{p_{n}-1}{p_{n}}, 1 }.
	\end{align}
	Therefore we see that
	\begin{align}
		z_k(\e_i\pi)
		&=
		\begin{cases}
			p_{n}-1 & \text{ if } 1 \leq k=n, \\
			(p_{n}-1)-p_{n} & \text{ if } k=n \leq 0,\\
			q_k & \text{ if } k\neq n, \, 1 \leq k, \text{ and } n+1 \leq k \leq m,\\
			p_k & \text{ if } k\neq n, \,  1 \leq k, \text{ and } k \leq n,\\
			q_k-p_k & \text{ if } k\neq n,\,  k \leq 0, \text{ and } n+1 \leq k \leq m,\\
			-p_k & \text{ if } k\neq n, \,  k \leq 0, \text{ and }  m+1 \leq k,\\
			0 & \text{ otherwise},
		\end{cases}\\
		&= z_k(\pi)-\delta_{k, n}.
	\end{align}
	Hence we obtain  $\e_{i}(\pi)=\e_{i}\mor(\pi)$, as desired.

	{\bf{Case 4}}. Assume that $i=i'(m)$ and $m+n$ is odd.
	Note that the function $H_{i}^\pi(t)$ attains a minimal value
	at $t=0$,  $t=1$,  and $t=\tpq{k}$, $k=m-1, m-3, \ldots, n+2$.
	By Proposition \ref{prop.sigminval}, we get
	\begin{align}
		\ep_{i}(\pi)
		&=-\min
			\left\{H_{i}^\pi(t) \, \middle| \,
				t \in \left\{
					0, \dpq{m-1}, \dpq{m-3}, \ldots,  \dpq{n+2}, 1
				\right\}
			\right\}\\
		&=\max
			\left\{-H_{i}^\pi(t) \, \middle| \,
				t \in \left\{
					0, \dpq{m-1}, \dpq{m-3}, \ldots,  \dpq{n+2}, 1
				\right\}
			\right\}\\
		&=\ep_{i}(\mor(\pi)).
	\end{align}
	We can show in exactly the same way as Case 3 that $\mor(\e_{i}\pi)=\e_{i}\mor(\pi)$.

	Let $\pi \in \B_1(\lambda)$, and $i\in I$.
	Because $\wt(\pi) =\wt(\mor(\pi))$ and $\ep_i(\pi) =\ep_i(\mor(\pi)) $,
	we have $\ph_i(\pi) =\ph_i(\mor(\pi)) $.
	Also, since both $\B_1(\lambda)$ and $\img(\Psi^+_{\iota^+}) \otimes \CT_\lambda \otimes \img(\Psi^-_{\iota^-})$ are normal crystals,
	and since $\mor(\e_i\pi) =\e_i\mor(\pi) $,
	we see that  $\mor(\f_i\pi) =\f_i\mor(\pi)$.
	This completes the proof of Theorem \ref{thm.mor}.
\end{proof}

\subsection{Proof of Theorem \ref{thm.mainb}.}\label{sec.mainprf2}
In this subsection, we assume that  $\lambda \in P$ is of the form either
$k_1\Lambda_1-\Lambda_2$ with $1\leq k_1 < a_1-1$ or
$\Lambda_1-k_2\Lambda_2$ with $1<k_2\leq a_2-1 $;
note that $\lambda$ satisfies the condition that $W\lambda \cap (P^+ \cup -P^+) = \emptyset$
(see \S \ref{sec.maintheorems}).
%

We can prove Theorem \ref{thm.mainb} in exactly the same way as  \cite[Theorem 3.2]{SY}.
So, we give only a sketch of the proof.
In the following, we assume that
$\lambda=k_1\Lambda_1-\Lambda_2$ with $1\leq k_1 < a_1-1$;
the proof for the case that
$\lambda=\Lambda_1-k_2\Lambda_2$ with $1<k_2\leq a_2-1 $
is similar.
Let us identify $\CB(\lambda)$ with
$\{b\in \CB(\infty) \otimes \CT_\mu \otimes \CB(-\infty)
\mid b^\ast \text{ is extremal} \}$ by Theorem \ref{thm.cbext}.
\begin{lemma}[{cf. \cite[Lemmas 3.7 and 3.8]{SY}}]\label{lem.cbconlem}
	\begin{enumerate}[\upshape(1)]
		\item Let $i \in I$ and $b \in \CB(\lambda)$ be such that $\e_ib \neq \0$.
		If $b$ is of the form $b=b_1 \otimes t_\lambda \otimes u_{-\infty}$ with $b_1 \neq u_{\infty}$,
		then $\e_i b=\e_i b_1 \otimes t_\lambda \otimes u_{-\infty}$.\label{item.lemcbconlem1}
		\item Let $i \in I$ and $b \in \CB(\lambda)$ be such that $\f_ib \neq \0$.
		If $b$ is of the form $b=u_{\infty} \otimes t_\lambda \otimes b_2$ with $b_2 \neq u_{-\infty}$,
		then $\f_i b=u_{\infty} \otimes t_\lambda \otimes \f_i b_2$.
	\end{enumerate}
\end{lemma}
\begin{proof}
We give a proof only for part \eqref{item.lemcbconlem1}.
Suppose, for a contradiction, that  $\e_ib=b_1 \otimes t_\lambda \otimes \e_iu_{-\infty}$.
We see by \eqref{eq.starope} that
$(\e_ib)^\ast = b_1^\ast \otimes t_{-\lambda- \wt (b_1)- \alpha_i} \otimes \e_i u_{-\infty}.$
Since $\ph_i((\e_ib)^\ast)\geq \ph_i(\e_i u_{- \infty})=1$,
it follows from the tensor product rule of crystals that
$\f_i(\e_ib)^\ast \neq \mathbf{0}$.
Because $\e_ib \in \mathcal{B}(\lambda)$,
we see that
$(\e_ib)^\ast $ is an extremal element of weight  $-\lambda$.
Since
$\pair{\wt(S_{\mathrm{id}}(\e_i b)^\ast)}{\alc_1}=\pair{-\lambda}{\al_1^\vee}=-k_1 \leq 0$,
we obtain
$\f_1(\e_ib)^\ast=\mathbf{0}$.
Therefore we have $i=2$ and
$(\e_2b)^\ast =  b_1^\ast \otimes t_{-\lambda- \wt (b_1)- \alpha_2} \otimes \e_2 u_{-\infty}$.
Because $\pair{\wt(S_\mathrm{id}(\e_2 b)^\ast)}{\alc_2}
= \langle -\lambda , \alpha_2^\vee \rangle = 1 \geq 0$, and
$(\e_2b)^\ast $ is an extremal element of extremal weight  $-\lambda$,
we see that $\e_2(\e_2b)^\ast=\0$, and hence $\ep_2((\e_2b)^\ast)=0$.
Since $\ep_2((\e_2 b)^\ast) \geq \ep_2(b_1^\ast)$,
we have $ \ep_2(b_1^\ast) = 0$, which implies $\ep_1(b_1^\ast) \geq 1$
because $b_1 \neq u_{\infty}$.
Hence
\begin{align}
\ph_2(b_1^\ast \otimes t_{-\lambda- \wt (b_1)- \alpha_2} )
&=\ph_2(b_1^\ast)+
\pair{-\lambda- \wt (b_1)- \alpha_2}{\alc_2}\\
&=(\ep_2(b_1^\ast)+
\pair{\wt(b_1^\ast )}{\alc_2} )+
\pair{-\lambda- \wt (b_1)- \alpha_2}{\alc_2} \\
&=\ep_2(b_1^\ast)+ \langle -\lambda-\alpha_2 , \alpha_2^\vee \rangle =-1.\label{eq.lemLScon1}
\end{align}
By this equality and
$\ep_2(\e_2 u_{-\infty})
=\ph_2(\e_2 u_{-\infty})- \pair{\wt(\e_2 u_{-\infty})}{\alc_2} =-1$,
it follows from the tensor product rule of crystals that
$	S_2(\e_2b)^\ast=\f_2(\e_2b)^\ast=
	b_1^\ast \otimes t_{-\lambda- \wt (b_1)- \al_2} \otimes u_{-\infty}$.
Since $\ep_1(b_1^\ast) \geq 1$, we obtain
$\e_1b_1^\ast \neq \0$.
Therefore
it follows from the tensor product rule of crystals that
$\ep_1(S_2(\e_2b)^\ast) \geq \ep_1(b_1^\ast)\geq 1$, that is,
$\e_1 S_2(\e_2b)^\ast \neq \0$.
However,
since   $(\e_2b)^\ast $ is an extremal element of  weight  $-\lambda$
and
$\pair{\wt(S_2(\e_2 b)^\ast)}{\alc_1}= \pair{ s_2(-\lambda)}{\alc_1}\geq 0$,
we see that  $\e_1 S_2(\e_2b)^\ast=\0$, which is a  contradiction.
\end{proof}
Lemma \ref{lem.cbconlem} implies the following proposition (see \cite[Proposition 3.9]{SY}).
\begin{proposition}\label{prop.cbconpro}
	It holds that
$\CB(\lambda) \subset (\CB(\infty) \otimes t_\lambda \otimes u_{-\infty})
\cup (u_{\infty} \otimes  t_\lambda \otimes  \CB(-\infty))$.
\end{proposition}
Here, we set $|\al|:=\sum_{i\in I} |c_i|$ for
$\al= \sum_{i\in I} c_i\al_i \in \bigoplus_{i \in I}\Z \al_i$.
By Proposition \ref{prop.cbconpro},
we see that $b \in  \CB(\lambda) $ is of the form either
$b=b_1 \otimes t_\lambda \otimes u_{-\infty} $ with some $b_1 \in \CB(\infty)$
or $b=u_{\infty}  \otimes t_\lambda \otimes b_2 $ with some $b_2 \in \CB(-\infty)$.
We deduce by induction on $|\wt(b_1)|$ \resp{$|\wt(b_2)|$} that
if $b$ is of the form $b_1 \otimes t_\lambda \otimes u_{-\infty} $
\resp{$b=u_{\infty}  \otimes t_\lambda \otimes b_2 $},
then
$b=\f_{i_r} \cdots \f_{i_1}u_\lambda$ \resp{$b=\e_{i_r} \cdots \e_{i_1}u_\lambda$}
for some $i_1, \ldots, i_r$ (see \cite[Proof of Theorem 3.2]{SY}).
Thus we have proved Theorem \ref{thm.mainb}.

\begin{remark}\label{rem.cbdecomp}
	Set
	\begin{align}
		\CB(\lambda)_-
		&:=
		\{
			\f_{i_l} \cdots \f_{i_2} \f_{i_1} u_\lambda
			\mid
			i_1, i_2, \ldots i_l \in I, l \geq 1
		\}
		\backslash \{ \0 \},\\
		\CB(\lambda)_+
		&:=
		\{
			\e_{i_l} \cdots \e_{i_2} \e_{i_1} u_\lambda
			\mid
			i_1, i_2, \ldots i_l \in I, l \geq 1
		\}
		\backslash \{ \0 \}.
	\end{align}
	By Theorem \ref{thm.mainb},
	we can decompose $\CB(\lambda)$ as
	\begin{align}
		\CB(\lambda)=\CB(\lambda)_- \sqcup \{ u_{\lambda} \} \sqcup 	\CB(\lambda)_+.
	\end{align}
	In particular, we see that $\#\CB(\lambda)_{\mu}<\infty$,
	where  $\CB(\lambda)_{\mu}:= \{ b \in \CB(\lambda) \mid \wt(b)=\mu \}$.
\end{remark}

\section{Computation of  $\#\CB(\Lambda_1 -\Lambda_2)_{\mu}$. }\label{sec.comp}
In this section, we assume that
\begin{align}\label{eq.symcal}
	A=
	\begin{pmatrix}
		2 & -a \\-a & 2\\
	\end{pmatrix} \
	\text{ with }
	a\geq 3, \ \text{ and } \lambda=\Lambda_1 -\Lambda_2.
\end{align}
By Corollary \ref{cor.main} (see also \cite[Theorem 3.6]{SY}),
we have $ \CB(\lambda) \cong  \B(\lambda)$.
The aim of this section is to give an algorithm for computing
$\mathrm{dim} V(\lambda)_\mu =\# \CB(\lambda)_{\mu}$ for $\mu\in P$.

\subsection{Subsets  of $\Z_{\geq 0}^{+\infty}$  and  $\Z_{\leq 0}^{-\infty}$.}\label{sec.alg}
We define maps $F$ and $F'$ as follows.
Recall that the sequence $\{ p_m\}$ is defined by recursive formulas \eqref{eq.pm} and \eqref{eq.pm2} for $\lambda=\Lambda_1 -\Lambda_2$;
we know from \cite[Lemma 3.5]{H} that
\begin{align}
	\cdots p_3>p_2> p_1=1 = p_0 =1< p_{-1} < p_{-2}< \cdots.
\end{align}
For $x\in \Z_{\geq 2}$,
we denote by  $n(x)$ the (unique) positive integer such that  $p_{n(x)-1}< x \leq p_{n(x)}$,
and set $n(1):=1$.
For  $x\in \Z_{\geq 1}$, we define  $F(x)$  to be  the (unique)  integer such that
\begin{align}\label{eq.notef}
	\frac{F(x)}{p_{n(x)+1}} \leq \frac{x}{p_{n(x)}} < \frac{F(x)+1}{p_{n(x)+1}},
\end{align}
and set $F(0):=0$;
note that
\begin{align}\label{eq.notefnpn}
	n(p_m)=m \text{ and }  F(p_m)=p_{m+1} \quad \text{ for some } m \geq 1.
\end{align}
Similarly, for  $x\in \Z_{\leq -2}$,
we denote by  $n'(x)$ the negative integer such that  $-p_{n'(x)} \leq x < -p_{n'(x)+1}$,
and set $n'(-1):=0$.
For $x\in \Z_{\leq -1}$, we define $F'(x)$ to be the  integer such that
\begin{align}
	\frac{F'(x)-1}{p_{n'(x)-1}} < \frac{x}{p_{n'(x)}} \leq  \frac{F'(x)}{p_{n'(x)-1}},
\end{align}
and set $F'(0):=0$;
note that
\begin{align}\label{eq.notef2}
	n'(-p_m)=m \text{ and }  F'(-p_m)=-p_{m-1} \quad \text{ for some } m \leq 0.
\end{align}
%
\begin{lemma}\label{lem.mainlem}
	\begin{enumerate}[\upshape(1)]
			\item  Let $x\in \Z_{\geq 1}$.	For all $m\geq n(x)$, we have
			\begin{align}\label{eq.mainlem}
				\frac{F(x)}{p_{m+1}} \leq \frac{x}{p_{m}} < \frac{F(x)+1}{p_{m+1}}.
			\end{align}\label{enu.mainlem1}
		\item Let $x \in \Z_{\leq -1}$. For all $m\leq n'(x)$, we have
			\begin{align}\label{eq.mainlem2}
					\frac{F'(x)-1}{p_{m-1}} < \frac{x}{p_{m}} \leq  \frac{F'(x)}{p_{m-1}}.
			\end{align}\label{enu.mainlem2}
	\end{enumerate}
\end{lemma}
\begin{proof}
	We give a proof only for part
\eqref{enu.mainlem1};
the proof for part \eqref{enu.mainlem2} is similar.
We show the following inequalities, which are equivalent to \eqref{eq.mainlem}:
\begin{align}
	\frac{F(x)}{x}  \leq \frac{p_{m+1}}{p_{m}}<\frac{F(x)+1}{x}.
\end{align}
Because the sequence
$\{ {p_{m+1}}/{p_{m}} \}_{m \geq 1}$ is an increasing sequence, the first inequality is obvious by the definition \eqref{eq.notef} of $F(x)$.
We show the second inequality.
Suppose, for a contradiction, that there exists $m \geq n(x)$ such that
\begin{align}\label{eq.contra}
	 \frac{p_{m+1}}{p_{m}} < \frac{F(x)+1}{x} \leq \frac{p_{m+2}}{p_{m+1}}.
\end{align}
We compute
\begin{align}
	p_m(F(x)+1)-xp_{m+1}
	& \leq p_m \left( \frac{p_{m+2}}{p_{m+1}}x \right) -xp_{m+1}
	\quad \text{ by the second inequality in \eqref{eq.contra} }\\
	&= \frac{x}{p_{m+1}}(p_mp_{m+2}-{p_{m+1}}^2) \\
	&= \frac{x}{p_{m+1}}(p_0 p_{2}-{p_{1}}^2)
		\quad \text{ by } \eqref{eq.pm}\\
	&= \frac{x}{p_{m+1}}(1 (a-1)-{1}^2)\\
	&\leq \frac{p_{n(x)}}{p_{n(x)+1}}(a-2)
			\quad  \text{ by }  n(x) \leq m \text{ and } x\leq p_{n(x)}\\
	&\leq \frac{p_1}{p_2}(a-2)
		\quad\text{ since } \{ {p_{m}}/{p_{m+1}} \}_{m \geq 1} \text{ is decreasing}\\
	&=\frac{1}{a-1}(a-2)<1 \quad  \text{ since } a\geq3. \label{eq.mainlem1}
\end{align}
	However,
	by the first  inequality in \eqref{eq.contra}, we obtain $p_m(F(x)+1)-xp_{m+1} > 0$,
	which contradicts the fact that
		$p_m(F(x)+1)-xp_{m+1}$ is a integer.
	Thus we have proved \eqref{eq.mainlem}.
\end{proof}
The next corollary follows from Lemma \ref{lem.mainlem} and the facts that
\begin{align}
	\lim_{m \to +\infty} \frac{p_{m+1}}{p_{m}} = \frac{a+ \sqrt{a^2-4}}{2}
	\quad \text{ and } \quad
	\lim_{m \to -\infty}\frac{p_{m-1}}{p_{m}} =\frac{a+ \sqrt{a^2-4}}{2}.
\end{align}
\begin{corollary}\label{cor.lim}
	For $x \in \Z_{\geq 0} $,
	\begin{align}
		F(x)= \left\lfloor   \frac{a+\sqrt{a^2-4}}{2} x  \right\rfloor,
	\end{align}
 and for $x \in \Z_{\leq 0} $,
	\begin{align}
		F'(x)= \left\lceil   \frac{a+\sqrt{a^2-4}}{2} x \right\rceil.
	\end{align}
\end{corollary}
Now, we set
\begin{align}
	Y_{-}&:=\{ (\ldots, y_j, \ldots, y_2, y_1) \in \Z_{\geq 0}^{+\infty}
		\mid y_{j+1} \leq F(y_j) \text{ for all } j \geq 1 \},\\
	Y_{+}&:=\{ (y_0, y_{-1}, \ldots, y_j, \ldots,  ) \in \Z_{\leq 0}^{-\infty}
		\mid F'(y_j) \leq  y_{j-1}  \text{ for all } j \leq 0 \}.
\end{align}
\begin{remark}\label{rem.Ypm}
	Let $y=(\ldots, y_j, \ldots y_2, y_1)  \in Y_{-}$.
	If there exists $l  \in \Z_{\geq1} $ such that $y_l=0$, then we see by  $F(0) =0$ that
	$y_j=0$ for all $j>l$.
	Therefore, if  $y\neq ( \ldots 0, 0)$,
	then
	$y$ is of the form $y=(\ldots, 0,  y_m, \ldots y_2, y_1) $ for some $m \in \Z_{\geq1}$,
	where $y_j>0$ for all $1 \leq j\leq m$.
	Similarly, we see that
	$ y \in Y_{+}$ is of the form  either
	$y=(0, 0, \ldots )$ or
	$y=(y_0, y_{-1}, \ldots, y_m, 0, \ldots,  )$ for some $m \in \Z_{\leq 0}$, where $y_j<0$ for all $m \leq j\leq 0$.
\end{remark}
For $n_1, n_2 \geq 0$ and $ 0 \leq m\leq n_1$, we set
\begin{align}
	Y_{-}(n_1, n_2; m):=\bigset{(\ldots,  y_j, \ldots, y_2, y_1) \in Y_{-} }{ y_1=m,  \sum_{j: \text{ odd}} y_j= n_1,
	\sum_{j: \text{ even}} y_j= n_2},
\end{align}
and for $n_1, n_2 \leq 0$ and $n_2 \leq m\leq 0$, we set
\begin{align}
	Y_{+}(n_1, n_2; m):=\bigset{(y_0, y_{-1}, \ldots,  y_j, \ldots ) \in Y_{+} }{ y_0=m,  \sum_{j: \text{ odd}} y_j= n_1,
	\sum_{j: \text{ even}} y_j= n_2}.
\end{align}

We give an algorithm for computing $\#Y_{-}(n_1, n_2; m)$.
It is obvious by Remark \ref{rem.Ypm} that
for $n_1, n_2 \geq 0$,
\begin{align}
	\#Y_{-}(n_1,n_2; 0)=
		\begin{cases}
			1 &\text{ if } n_1 = n_2=0,\\
			0 &\text{ otherwise}.
		\end{cases}
\end{align}
Assume that $1\leq m \leq n_1$.
We have  $Y_{-}(m, 0; m)=\{(\ldots, 0, m) \}$
and  $Y_{-}(n_1, 0; m)=\emptyset$ if $m < n_1$ by Remark \ref{rem.Ypm}.
Hence, for  $1\leq m \leq n_1$,
\begin{align}\label{eq.ym0}
	\#Y_{-}(n_1,0;m)=
		\begin{cases}
			1 &\text{ if } m = n_1,\\
			0 &\text{ if } m < n_1.
		\end{cases}
\end{align}
Assume that $1\leq m\leq n_1$ and $n_2 \geq 1$.
Let $y= (\ldots,  y_3,  y_2, y_1) \in Y_{-}(n_1,n_2;m)$.
Then, $y$ is of the form  $y= (\ldots,  y_3,  y_2, m)$.
By Remark \ref{rem.Ypm} and the assumption that $n_2 \geq 1$, we have $y_2\geq 1$.
By the definition of $Y_{-}$, we obtain $y_2 \leq F(y_1)=F(m)$.
Moreover, we have  $y_2\leq n_2$
since $y \in Y_{-}(n_1,n_2;m)$.
Hence, $1 \leq y_2\leq \min\{ F(m), n_2\}$.
Then we have
\begin{align}
		Y_{-}(n_1,n_2;m)=\bigsqcup_{l=1}^{\min\{F(m), n_2\}}\bigset{(\ldots,  y_3, l ,m) \in Y_{-} }{ m+
	 \sum_{j=3, 5, \ldots} y_j= n_1,
	l+ \sum_{j=4, 6 \ldots} y_j= n_2}.
\end{align}
For $1 \leq  l\leq \min\{ F(m), n_2\}$,
we see that
\begin{align}
	&\# \bigset{(\ldots, y_4,   y_3, l ,m) \in Y_{-} }{ m+ \sum_{j=3, 5, \ldots} y_j= n_1, \,
	l+\sum_{j=4, 6 \ldots} y_j= n_2}.\\
	=& \#\bigset{(\ldots, y_4,  y_3,  l) \in Y_{-} }{ \sum_{j=3, 5, \ldots} y_j= n_1-m, \,
			l+ \sum_{j=4, 6 \ldots } y_j= n_2}.\\
	=&\#Y_{-}(n_2,n_1-m;l).
\end{align}
Hence we obtain
\begin{align}\label{eq.ygen}
	\#Y_{-}(n_1,n_2;m)=\sum_{l=1}^{\min\{F(m),n_2\}} \#Y_{-}(n_2,n_1-m;l).
\end{align}
If $n_1-m=0$,  then
we see by \eqref{eq.ym0} that  $\#Y_{-}(n_2,n_1-m;l)$ becomes a finite sum of $0$ and $1$.
Assume that $n_1-m>0$.
We set $n_1':=n_2$ and  $n_2':=n_1-m$.
Let $m'$ be such that  $1 \leq  m'\leq \min\{ F(m), n_2\}$.
Since $1  \leq m' \leq n_1'$ and $n_2'\geq 1$,
we obtain,
by the  same argument as above,
\begin{align}
	\#Y_{-}(n_1',n_2';m')
	&=\sum_{l=1}^{\min\{F(m'), n_2'\}} \#Y_{-}(n_2', n_1'-m', l)\\
	&= \sum_{l=1}^{\min\{F(m'), n_1-m\}} \#Y_{-}(n_1-m, n_2-m', l).
\end{align}
Because $n_2-m'\leq n_2-1$,
this  process ends after at most $n_2$ steps.
Similarly, we can compute  $\#Y_{+}(n_1,n_2;m)$.
\subsection{Number of elements  in $\CB(\lambda)_{\mu}$.}
We set
\begin{align}
	Z(\lambda)_{-}&:=
	\{ y \in \Z_{\iota^+}^{+\infty} \mid  y \text{ is of the form either }  \eqref{enu.z1} \text{ or } \eqref{enu.z2}
	\},\\
	Z(\lambda)_{+}&:=
	\{ y \in \Z_{\iota^-}^{-\infty} \mid  y \text{ is of the form either }  \eqref{enu.z5} \text{ or } \eqref{enu.z6}
	\},
\end{align}
where
\begin{enumerate}[\upshape(a)]
	\item $(\ldots, 0, p_m, \ldots, p_2, p_1)$ for some  $1 \leq m$;\label{enu.z1}
	\item $(\ldots, 0, q_m, \ldots, q_{n+2}, q_{n+1}, p_n, \ldots, p_2, p_1) $ for some $n, m\in\Z $ such that $1 \leq n < m$, where $q_{m}, q_{m-1}, \ldots, q_{n+1}$ are  integers satisfying that
	$0<q_j<p_j$ for $n+1\leq j \leq m$, and
  $\tpq{j+1}<\tpq{j}$  for $n+1\leq j \leq m-1$;\label{enu.z2}
	\item $(-p_0, -p_{-1}, \ldots, -p_{m+1},  q_m-p_m, q_{m-1}-p_{m-1}, \ldots, q_{n+1}-p_{n+1}, 0, \dots) $ for some $n, m \in\Z $ such that $n < m \leq -1$, where $q_{m}, q_{m-1}, \ldots, q_{n+1}$ are integers satisfying that
	$0<q_j<p_j$ for $n+1\leq j \leq m$,  and
  $\tpq{j+1}<\tpq{j}$  for $n+1\leq j \leq m-1$;\label{enu.z5}
	\item $(-p_0, -p_{-1},\ldots, -p_{n+1}, 0, \dots)$ for some $n \leq -1$.\label{enu.z6}
\end{enumerate}
Let
\begin{align}
	Y(\lambda)_{-}&:=\{ (\ldots, 0, y_m, \ldots y_2, y_1) \in Y_{-} \mid y_1 = 1 \},\\
	Y(\lambda)_{+}&:=\{ (y_0, y_{-1}, \ldots, y_j, \ldots,  ) \in Y_{+} \mid y_0 =-1 \}.
\end{align}

\begin{proposition}\label{prop.zy}
	It hold that $Z(\lambda)_{-} =Y(\lambda)_{-}$ and   $Z(\lambda)_{+} =Y(\lambda)_{+}$.
\end{proposition}
\begin{proof}
	We give a proof only for  $Z(\lambda)_{-} =Y(\lambda)_{-}$;
	the proof for $Z(\lambda)_{+} =Y(\lambda)_{+}$ is similar.
	By \eqref{eq.notefnpn} and Lemma \ref{lem.mainlem},
	we can easily check that  $Z(\lambda)_{-} \subset Y(\lambda)_{-}$.
	We show the reverse inclusion  $Z(\lambda)_{-} \supset Y(\lambda)_{-}$.
	Let $y \in Y(\lambda)_{-}$.
	By Remark \ref{rem.Ypm} and the definition of $Y(\lambda)_{-}$,
	the element $y$ is of the form $y=(\ldots, 0,  y_m, \ldots y_2, y_1) $ for some $m\in \Z_{\geq 1}$ with  $y_j>0$ for all $1 \leq j\leq m$ .
	By $y_1= 1 =  p_1$ and \eqref{eq.notefnpn},
	we see that $y_2\leq F(y_1) \leq  F(p_1)=p_2$,
	where we use the monotonicity of $F$ (see Corollary \ref{cor.lim}).
	Similarly, we see by $y_2 \leq p_2$ and  \eqref{eq.notefnpn} that
	$y_3\leq F(y_2)\leq  F(p_2)=p_3$.
	Repeating this argument, we obtain
	\begin{equation}\label{eq.zy1}
			1\leq y_j \leq p_j  \quad \text{ for  all } 1 \leq j\leq m.
	\end{equation}
	Let $n$ be the largest integer $n' \leq m$ such that $y_{n'}=p_{n'}$;
	note that $y_1=1=p_1$.
	Since $y \in Y_-$ and
	\eqref{eq.zy1},
	we have $p_{n}\leq F(y_{n -1}) \leq F(p_{n -1})=p_{n}$.
	Hence we  obtain  $F(y_{n-1})=p_{n}$.
	Since $1=F(y_{n-1 })/p_{n} \leq y_{n-1}/p_{n-1}$ by  Lemma \ref{lem.mainlem}
	and this equality,
	we get $p_{n-1} \leq y_{n-1}$.
	Hence, by \eqref{eq.zy1}, we have $p_{n-1} =y_{n-1}$.
	Repeating this argument, we obtain
	$p_{j} =y_{j}$ for  all $1 \leq j \leq n$.
	If $n=m$, then we obtain $y=(\ldots, 0,  p_m, \ldots p_2, p_1)$,
	which  is of the form \eqref{enu.z1}, and hence
	 $y \in  Z(\lambda)_{-}$.
	Assume that $n < m$.
	Then we have $y=(\ldots, y_m, \ldots, y_{n+1}, p_n, \ldots p_1)$, where
	\begin{align}\label{eq.zy2}
		1 \leq y_j < p_j \text{ for all } n+1 \leq j \leq m.
 	\end{align}
	Since  $y \in  Y_{-}$, we see by  Lemma \ref{lem.mainlem} that
	$y_{j+1}/p_{j+1}\leq y_{j}/p_{j}$ for all $n+1 \leq j \leq m-1$.
	Suppose, for a contradiction, that $y_{j+1}/p_{j+1}= y_{j}/p_{j}$.
	Then, $y_{j+1}= y_{j}p_{j+1}/p_{j}$.
	We see by $p_0=p_1=1$ and \eqref{eq.pm} that $p_j$ and $p_{j+1}$ are relatively prime
	(see \cite[Lemma 4.5 (1)]{H}).
	Because $y_{j+1}$ is a positive integer, we obtain $y_j \geq p_j$,
	which contradicts  \eqref{eq.zy2}.
	Therefore, we obtain $y_{j+1}/p_{j+1}< y_{j}/p_{j}$ for all $n+1 \leq j \leq m-1$.
	Thus we see that $y$ is of the form \eqref{enu.z2}, and hence
	$y \in  Z(\lambda)_{-}$.
	Thus we have proved the proposition.
\end{proof}

Let  $\mu=\lambda -n_1\al_1-n_2\al_2$
with $n_1, n_2 \in \Z_{\geq 0}$, and assume that   $\mu\neq \lambda$.
By the results of Section \ref{sec.mainprf}
and the fact that $p_1=1$ (and hence there is no integer $q$ such that $0 < q <p_1$),
there exists a natural bijection from
$\CB(\lambda)_\mu$
onto
\begin{align}
	\bigset{(\ldots, 0, y_m, \ldots y_2, y_1)
	\in Z(\lambda)_{-}
	}{
	\sum_{j: \text{ odd}} y_j=n_1, \sum_{j: \text{ even}}y_j=n_2}.
\end{align}
Moreover,
we see that
\begin{align}
		Y_{-}(n_1,n_2;1)
		=& \bigset{(\ldots, 0, y_m, \ldots y_2, y_1)
		\in Y_{-}
		}{
		y_1=1,  \sum_{j: \text{ odd}} y_j=n_1, \sum_{j: \text{ even}}y_j=n_2}\\
		=& \bigset{(\ldots, 0, y_m, \ldots y_2, y_1)
		\in Y(\lambda)_{-}
		}{
		\sum_{j: \text{ odd}} y_j=n_1, \sum_{j: \text{ even}}y_j=n_2}.
\end{align}
By Proposition \ref{prop.zy}, we see that
 $\# \CB(\lambda)_{\mu}= \#Y_{-}(n_1,n_2;1)$.
Similarly, if $\mu=\lambda-n_1\al_1-n_2\al_2$ with $n_1, n_2 \in \Z_{\leq 0}$,
and $\mu \neq \lambda$, then $\#\CB(\lambda)_\mu=\#Y_+(n_1, n_2; -1)$.
Summarizing these, we obtain the following theorem.
\begin{theorem}
	For $\mu \in P$,
	it holds that
	\begin{align}
		\# \CB(\lambda)_{\mu}=
		\begin{cases}
			1 &  \text{\rm if } \mu=\lambda,\\
			 \#Y_{-}(n_1,n_2;1)
			 &  \text{\rm if }  \mu\neq\lambda \text{\rm{ and }} \mu=\lambda-n_1\al_1-n_2\al_2 \text{\rm{ for some }} n_1,n_2\geq 0,\\
			 \#Y_{+}(n_1,n_2;-1)
			 & \text{\rm if }  \mu\neq\lambda \text{\rm{ and }} \mu=\lambda-n_1\al_1-n_2\al_2 \text{\rm{ for some }} n_1,n_2 \leq 0,\\
			0
			 & \text{\rm otherwise}.
		\end{cases}
	\end{align}
\end{theorem}

\section*{Acknowledgment.}
The author would like to thank Daisuke Sagaki, who is his supervisor, for his kind support and advice.


\begin{thebibliography}{99}
\bibitem{H} R. Hiasa, Connectedness of Lakshmibai-Seshadri path crystals for hyperbolic Kac-Moody algebras of rank 2, Comm. Algebra, 49, 2021, pp.772--789.
\bibitem{HK} J. Hong and S.-J. Kang, Introduction to quantum groups and crystal bases, Graduate Studies in Mathematics, 42, Amer. Math. Soc., 2002.
\bibitem{INS} M. Ishii, S. Naito, and D. Sagaki, Semi-infinite Lakshmibai-Seshadri path model for level-zero extremal weight modules over quantum affine algebras, Adv. Math., 2016, pp.967--1009.
\bibitem{J} A. Joseph, Quantum groups and their primitive ideals, Ergebnisse der Mathematik und ihrer Grenzgebiete (3), 29, Springer-Verlag, 1995.
\bibitem{K} M. Kashiwara, Crystal bases of modified quantized enveloping algebra, Duke Math. J., 73, 1994, pp.383--413.
\bibitem{Kdem} M. Kashiwara, The crystal base and Littelmann's refined Demazure character formula, Duke Math. J., 71, 1993, pp.839--858.
\bibitem{Kocb} M. Kashiwara, On crystal bases, in ``Representation of groups'', CMS Conf. Proc., 16, pp.155--197, Amer. Math. Soc., Providence, RI, 1995.
\bibitem{Ksim} M. Kashiwara, Similarity of crystal bases, in ``Lie algebras and their representation'', Contemp. Math., 194, pp.177--186, Amer. Math. Soc., Providence, RI, 1996.
\bibitem{NS1} S. Naito and D. Sagaki, Path model for a level-zero extremal weight module over a quantum affine algebra, Int. Math. Res. Not., 2003, pp.1731--1754.
\bibitem{NS2} S. Naito and D. Sagaki, Path model for a level-zero extremal weight module over a quantum affine algebra II, Adv. Math., 200, 2006, pp.102--124.
\bibitem{NZ} T. Nakashima and A. Zelevinsky, Polyhedral realizations of crystal bases for quantized Kac-Moody algebras, Adv. Math., 131, 1997, pp.253--278.
\bibitem{L2} P. Littelmann, A Littlewood-Richardson rule for symmetrizable Kac-Moody algebras, Invent. Math., 116, 1994, pp.329--346.
\bibitem{L} P. Littelmann, Paths and root operators in representation theory, Ann. of Math., (2), 142, 1995, pp.499--525.
\bibitem{SY} D. Sagaki and D. Yu, Path model for an extremal weight module over the quantized hyperbolic Kac-Moody algebra of rank 2, Comm. Algebra, 49, 2021, pp.690--705.
\bibitem{Yu} D. Yu, Lakshmibai-Seshadri paths for hyperbolic Kac-Moody algebras of rank 2, Comm. Algebra, 46, 2018, pp.2702--2713.
\end{thebibliography}
\end{document}